\documentclass[12pt]{amsart}

\usepackage{amssymb, amsmath, amscd, newlfont, bbm, slashbox}
\usepackage[foot]{amsaddr}

\usepackage{hyperref} 
\hypersetup{
	colorlinks=true,        % false: boxed links; true: colored links
	linkcolor={black},  % color of internal links (change box color with linkbordercolor)
	citecolor={black},% color of links to bibliography
	urlcolor={black}     % color of external links	
}
\usepackage{enumitem}
\usepackage{tikz-cd}
\usepackage{comment}
\usepackage{cite, mathtools}

\newcommand{\spacedcdot}{{\,\cdot\,}}

\newcommand{\Q}{{\mathbb{Q}}}
\newcommand{\C}{{\mathbb{C}}}
\newcommand{\R}{{\mathbb{R}}}
\newcommand{\Z}{{\mathbb{Z}}}

\newcommand{\Lie}{{\mathrm{Lie}}}

\newcommand{\GL}{{\mathrm{GL}}}

\newcommand{\Sp}{{\mathrm{Sp}}}

\newcommand{\SL}{{\mathrm{SL}}}
\newcommand{\U}{{\mathrm{U}}}

\DeclareMathOperator{\artanh}{{\mathrm{artanh}}}

\DeclareMathOperator{\tr}{{\mathrm{tr}}}

\newcommand{\calD}{{\mathcal{D}}}

\newcommand{\calH}{{\mathcal{H}}}

\usepackage[mathscr]{eucal}

\providecommand{\abs}[1]{\left\lvert#1\right\rvert}
\providecommand{\norm}[1]{\left\lVert#1\right\rVert}
\providecommand{\scal}[2]{\left<#1,#2\right>}

\newtheorem{theorem}{Theorem}[section]
\newtheorem{lemma}[theorem]{Lemma}
\newtheorem{proposition}[theorem]{Proposition}
\newtheorem{corollary}[theorem]{Corollary}

\theoremstyle{definition}

\theoremstyle{remark}

\newtheorem*{remark*}{Remark}

\newtheoremstyle{named}{}{}{\itshape}{}{\bfseries}{.}{.5em}{#1 \thmnote{#3}}
\theoremstyle{named}

\numberwithin{equation}{section}

\title[Vector-valued Siegel Poincar\'e series]{Construction and non-vanishing of a family of vector-valued Siegel Poincar\'e series}

\author{Sonja \v Zunar}

\address{Faculty of Geodesy,
	University of Zagreb,
	Ka\v ci\'ceva 26,
	10000 Zagreb,
	Croatia}
\email{szunar@geof.hr}

\thanks{This work is supported by the Croatian Science Foundation under the project number HRZZ-IP-2022-10-4615.}

\begin{document}

\begin{abstract}
	Using Poincar\'e series of $ K $-finite matrix coefficients of integrable antiholomorphic discrete series representations of $ \Sp_{2n}(\R) $, we construct a spanning set for the space $ S_\rho(\Gamma) $ of Siegel cusp forms of weight $ \rho $ for $ \Gamma $, where $ \rho $ is an irreducible polynomial representation of $ \GL_n(\C) $ of highest weight $ \omega\in\Z^n $ with $ \omega_1\geq\ldots\geq\omega_n>2n $, and $ \Gamma $ is a discrete subgroup of $ \Sp_{2n}(\R) $ commensurable with $ \Sp_{2n}(\Z) $. Moreover, using a variant of Mui\'c's integral non-vanishing criterion for Poincar\'e series on unimodular locally compact Hausdorff groups, we prove a result on the non-vanishing of constructed Siegel Poincar\'e series. 
\end{abstract}

\maketitle

\section{Introduction}

	Although Siegel modular forms have been recognized as a natural generalization of elliptic modular forms since the 1930s \cite{siegel35, siegel39}, the elementary theory of Siegel modular forms, especially the vector-valued ones, remains full of open questions (cf.\ \cite{van_der_geer08}). In recent years, there has been significant progress in constructing examples of vector-valued Siegel modular forms \cite{
		ramakrishnan_shahidi07, 
		clery_van_der_geer15, clery_faber_van_der_geer20}, computing dimensions \cite{taibi17} and proving lifting theorems \cite{chenevier_lannes19, atobe19}
		% and structure theorems \cite{ibukiyama12} 
		for spaces of vector-valued Siegel modular forms, as well as in studying critical L-values of vector-valued Siegel cusp forms \cite{kozima00, pitale_saha_schmidt21, horinaga_pitale_saha_schmidt22}. 
	
	In this paper, by adapting representation-theoretic methods developed by G.\ Mui\'c in his work on elliptic modular forms \cite{muic10}, we construct a family of Siegel Poincar\'e series that spans the space $ S_\rho(\Gamma) $ of Siegel cusp forms of weight $ \rho $ for $ \Gamma $, where $ (\rho,V) $ is an irreducible polynomial representation of $ \GL_n(\C) $ of highest weight $ \omega\in\Z^n $ with $ \omega_1\geq\ldots\geq\omega_n>2n $, and $ \Gamma $ is a discrete subgroup of $ \Sp_{2n}(\R) $ that is commensurable with $ \Sp_{2n}(\Z) $, i.e., is such that $ \Gamma\cap\Sp_{2n}(\Z) $ is of finite index in both $ \Gamma $ and $ \Sp_{2n}(\Z) $. In this way, we generalize results of \cite{zunar23} from the scalar-valued case (i.e., the case when $ \rho=\det^m $ for some $ m\in\Z_{>2n} $) to the vector-valued case. This generalization is not straightforward: Quite a few subtleties arise from the fact that the structure of a general irreducible polynomial representation $ \rho $ of $ \GL_n(\C) $ is more complicated than that of a character (see \S \ref{subsec:012} and, e.g., the proof of Lem.\,\ref{lem:026}). Moreover, the intricacies arising from lifting vector-valued Siegel cusp forms to cuspidal automorphic forms on $ \Sp_{2n}(\R) $ are more involved than in the scalar-valued case.

	Let us briefly present our main results. Let $ \calH_n $ denote the Siegel upper half-space of degree $ n\in\Z_{>0} $, i.e., the set of symmetric matrices $ z=x+iy\in M_n(\C) $ whose imaginary part $ y $ is positive definite. Let $ (\rho,V) $ and $ \Gamma $ be as above. We fix a Hermitian inner product $ \scal\spacedcdot\spacedcdot $ on $ V $ with respect to which the restriction $ \rho\big|_{\mathrm{U}(n)} $ is unitary and denote by $ \norm\spacedcdot $ the corresponding norm on $ V $. The space $ S_\rho(\Gamma) $ of Siegel cusp forms for $ \Gamma $ of weight $ \rho $ can be described as the finite-dimensional complex vector space of holomorphic functions $ f:\calH_n\to V $ such that $ \sup_{z\in\calH_n}\norm{\rho\Big(y^{\frac12}\Big)f(z)}<\infty $ and
	\[ f\left((Az+B)(Cz+D)^{-1}\right)=\rho(Cz+D)f(z) \]
	for all $ \begin{pmatrix}A&B\\C&D\end{pmatrix}\in\Gamma $ and $ z\in\calH_n $ (see Lem.\,\ref{lem:026}).
	Let $ \C[X_{r,s}:1\leq r,s\leq n] $ denote the ring of polynomials with complex coefficients in the variables $ X_{r,s} $, where $ r,s\in\left\{1,2,\ldots,n\right\} $. We give a representation-theoretic proof of the following theorem (see Thm.\,\ref{thm:062}).
	
	\begin{theorem}\label{thm:100}
		Let $ n\in\Z_{>0} $. Let $ \Gamma $ be a subgroup of $ \Sp_{2n}(\R) $ commensurable with $ \Sp_{2n}(\Z) $. Let $ (\rho,V) $ be an irreducible polynomial representation of $ \GL_n(\C) $ of highest weight $ \omega\in\Z^n $, where $ \omega_1\geq\ldots\geq\omega_n>2n $. 
		Given $ \mu\in\C[X_{r,s}:1\leq r,s\leq n] $ and $ v\in V $, the Poincar\'e series
		\[ \begin{aligned}
			\left(P_{\Gamma,\rho}f_{\mu,v}\right)(z)=\sum_{\left(\begin{smallmatrix}A&B\\C&D\end{smallmatrix}\right)\in\Gamma}\Bigg(&\mu\bigg(\Big(Az+B-i(Cz+D)\Big)\Big(Az+B+i(Cz+D)\Big)^{-1}\bigg)\\
			&\cdot\rho\left(\frac1{2i}\Big(Az+B+i(Cz+D)\Big)\right)^{-1}v\,\Bigg)
		\end{aligned} \]
		converges absolutely and uniformly on compact subsets of $ \calH_n $. Moreover, we have
		\[ S_\rho(\Gamma)=\mathrm{span}_\C\left\{P_{\Gamma,\rho}f_{\mu,v}:\mu\in\C[X_{r,s}:1\leq r,s\leq n],\,v\in V\right\}. \]
	\end{theorem}

	To prove Thm.\,\ref{thm:100}, we use the classical lift that identifies each function $ f:\calH_n\to V $ with the function $ F_f:\Sp_{2n}(\R)\to V $,
	\[ F_f\left(\begin{pmatrix}A&B\\C&D\end{pmatrix}\right)=\rho(iC+D)^{-1}f\Big((iA+B)(iC+D)^{-1}\Big) \]
	(see also \eqref{eq:047}). Fixing a highest weight vector $ v^{top} $ for $ \rho $, we show that the assignment
	\begin{equation}\label{eq:102}
		f\mapsto \scal{F_f(\spacedcdot)}{v^{top}}
	\end{equation}
	defines a linear isomorphism $ \Lambda_{\rho,\Gamma} $ from $ S_\rho(\Gamma) $ onto a finite-dimensional subspace $ L_\rho(\Gamma) $ of the Hilbert space $ L^2(\Gamma\backslash\Sp_{2n}(\R)) $. More precisely, $ L_\rho(\Gamma) $ is contained in the subspace of $ K $-finite vectors of the $ \pi_\rho^* $-isotypic component $ L^2(\Gamma\backslash\Sp_{2n}(\R))_{[\pi_\rho^*]} $ of the right regular representation of $ \Sp_{2n}(\R) $ on $ L^2(\Gamma\backslash\Sp_{2n}(\R)) $, where $ \left(\pi_\rho^*,\calH^2(\rho)^*\right) $ is a certain integrable antiholomorphic discrete series representation of $ \Sp_{2n}(\R) $. Next, as the chief ingredient of our proof, we apply Mili\v ci\'c's result \cite[Lem.\,6.6]{muic19} on Poincar\'e series of $ K $-finite matrix coefficients of integrable discrete series representations of connected semisimple Lie groups with finite center. By this result, the subspace of $ K $-finite vectors in $ L^2(\Gamma\backslash\Sp_{2n}(\R))_{[\pi_\rho^*]} $ is spanned by the (absolutely and locally uniformly convergent) Poincar\'e series $ P_\Gamma c_{h,h'}:\Sp_{2n}(\R)\to\C $, 
	\[ \left(P_\Gamma c_{h,h'}\right)(g)=\sum_{\gamma\in\Gamma}c_{h,h'}(\gamma g), \]
	where $ h $ and $ h' $ run over the $ K $-finite vectors in $ \calH^2(\rho)^* $, and $ c_{h,h'} $ denotes the $ K $-finite matrix coefficient of $ \pi_\rho^* $ given by 
	\[ c_{h,h'}(g)=\scal{\pi_\rho^*(g)h}{h'}_{\calH^2(\rho)^*},\qquad g\in\Sp_{2n}(\R). \]
	In Cor.\,\ref{cor:106}, we single out a subfamily of matrix coefficients $ c_{h,h'} $ whose Poincar\'e series span $ L_\rho(\Gamma) $, and compute those matrix coefficients in Prop.\,\ref{prop:049} by following the idea of \cite[proof of Lem.\,3-5]{muic10} and employing a few classical Harish-Chandra's results in a suitable way. Thm.\,\ref{thm:100} follows by applying the inverse of the linear isomorphism $ \Lambda_{\rho,\Gamma} $ to the constructed spanning set for $ L_\rho(\Gamma) $.
	
	In Sec.\ \ref{sec:007}, we note that some of the Siegel cusp forms constructed in Thm.\,\ref{thm:100} are related to the well-known reproducing kernel function for $ S_\rho(\Gamma) $. More precisely, equipping the space $ S_\rho(\Gamma) $ with the Petersson inner product
	\[ \scal{f_1}{f_2}_{S_\rho(\Gamma)}=\frac1{\abs{\Gamma\cap\left\{\pm I_{2n}\right\}}}\int_{\Gamma\backslash\calH_n}\scal{\rho\Big(y^{\frac12}\Big)f_1(z)}{\rho\Big(y^{\frac12}\Big)f_2(z)}\,d\mathsf v(z), \]
	where $ d\mathsf v(x+iy)=\det y^{-n-1}\prod_{1\leq r\leq s\leq n}dx_{r,s}\,dy_{r,s}  $,
	we have the following lemma.
	
	\begin{lemma}\label{lem:107}
		Let $ n $, $ \Gamma $, and $ \rho $ be as in Thm.\,\ref{thm:100}. Then, there exists a constant $ c_\rho\in\C^\times $, depending only on $ \rho $, such that 
		\[ \scal{f(iI_n)}v=c_\rho\scal f{P_{\Gamma,\rho}f_{1,v}}_{S_\rho(\Gamma)},\qquad f\in S_\rho(\Gamma),\ v\in V.  \] 
	\end{lemma}
	
	Finally, in Sec.\ \ref{sec:008}, we turn our attention to the question of non-vanishing of the constructed Poincar\'e series and apply to the vector-valued automorphic forms $ F_{P_{\Gamma,\rho}f_{\mu,v}} $ on $ \Sp_{2n}(\R) $ a vector-valued variant, Prop.\,\ref{prop:063}, of Mui\'c's integral non-vanishing criterion \cite[Thm.\,4.1]{muic09} for Poincar\'e series on unimodular locally compact Hausdorff groups. As a result, in Thm.\,\ref{thm:101}, given a representation $ (\rho,V) $ of $ \GL_n(\C) $ as in Thm.\,\ref{thm:100}, a polynomial $ \mu\in\C[X_{r,s}:1\leq r,s\leq n]\setminus\left\{0\right\} $, and a vector $ v\in V\setminus\left\{0\right\} $, we find an integer $ N_0=N_0(\rho,\mu,v)\in\Z_{\geq3} $ such that for every $ N\in\Z_{\geq N_0} $ and for every subgroup $ \Gamma $ of finite index in $ \Gamma_n(N) $, we have that
	\[ F_{P_{\Gamma,\rho}f_{\mu,v}}\not\equiv0\qquad\text{and}\qquad P_{\Gamma,\rho}f_{\mu,v}\not\equiv0. \]
	Here, $ \Gamma_n(N) $ denotes the principal level $ N $ congruence subgroup of $ \Sp_{2n}(\Z) $, defined by
	\[ \Gamma_n(N)=\left\{\gamma\in\Sp_{2n}(\Z):\gamma\equiv I_{2n}\pmod N\right\}. \]

\section{Preliminaries}

\subsection{Basic notation}\label{subsec:110}

Throughout the paper, let us fix $ n\in\Z_{>0} $. Given a complex square matrix $ A $, we denote by $ \mathrm{tr}(A) $, $ A^\top $, $ \overline A $, and $ A^* $ the trace, the transpose, the conjugate, and the conjugate transpose of $ A $, respectively. Given $ S\subseteq\C $, we denote by $ \mathrm{Diag}_n(S) $ the set of diagonal matrices in $ M_n(\C) $ whose diagonal coefficients belong to $ S $. Let $ i\in\C $ be the imaginary unit.  Given a real Lie algebra $ \mathfrak g $, let $ \mathfrak g_\C $ denote the complex Lie algebra $ \mathfrak g\otimes_\R\C $, and let $ \mathcal Z(\mathfrak g) $ denote the center of the universal enveloping algebra $ \mathcal U(\mathfrak g) $ of $ \mathfrak g_\C $. 

Let $ G $ be a Lie group with Lie algebra $ \mathfrak g $. Given a subspace $ W $ of a left $ \mathfrak g_\C $-module $ V $ and a subset $ S $ of $ \mathfrak g_\C $, we denote by $ W^S $ the subspace of $ S $-invariants in $ W $, i.e.,
\begin{equation}\label{eq:113}
	 W^S=\left\{w\in W:s.w=0\text{ for all }s\in S\right\}.
\end{equation}
Given a maximal compact subgroup $ K $ of $ G $ and a unitary representation $ (\pi,H) $ of $ G $ on a complex Hilbert space $ H $, we denote by $ H_K $ the $ (\mathfrak g,K) $-module of smooth vectors $ v\in H $ that are $ K $-finite, i.e., satisfy $ \dim_\C\mathrm{span}_\C\pi(K)v<\infty $. Moreover, for every irreducible unitary representation $ \tau $ of $ G $, we denote by $ H_{[\tau]} $ the $ \tau $-isotypic component of $ H $, i.e., the closure of the sum of irreducible closed $ G $-invariant subspaces of $ H $ that are equivalent to $ \tau $. Finally, we write $ H^* $ for the dual Hilbert space of $ H $ and let $ \left(\pi^*,H^*\right) $ denote the contragredient representation of $ \pi $, defined by
\[ \pi^*(g)v^*=\left(\pi(g)v\right)^*,\qquad g\in G,\ v\in H, \]
where we use the notation $ v^*=\scal\spacedcdot v_H\in H^* $ for $ v\in H $.  For details on these representation-theoretic notions, see, e.g., \cite[Ch.\,I, \S3]{knapp86} or \cite[Ch.\,1 and \S3.3.5]{wallach88}.

\subsection{Real symplectic group} Let $ J_n=\begin{pmatrix}0&I_n\\-I_n&0\end{pmatrix}\in\GL_{2n}(\C) $. The real symplectic group
\[ \Sp_{2n}(\R)=\left\{g\in\GL_{2n}(\R):g^\top J_ng=J_n\right\} \]
acts transitively on the Siegel upper half-space 
\[ \calH_n=\left\{z=x+iy\in M_n(\C):z^\top=z\text{ and }y>0 \right\} \] 
by
\begin{equation}\label{eq:018}
	 g.z=(Az+B)(Cz+D)^{-1},\qquad g=\begin{pmatrix}A&B\\C&D\end{pmatrix}\in\Sp_{2n}(\R),\ z\in\calH_n. 
\end{equation}
For $ g $ and $ z $ as above, we have
\begin{equation}\label{eq:019}
	\Im(g.z)=(Cz+D)^{-*}y(Cz+D)^{-1}.
\end{equation}

The stabilizer of $ iI_n $ with respect to the action \eqref{eq:018} is the maximal compact subgroup 
\[ K=\Sp_{2n}(\R)\cap \U(2n)=\left\{k_{A+iB}\coloneqq\begin{pmatrix}A&B\\-B&A\end{pmatrix}\in\GL_{2n}(\R):A+iB\in\U(n)\right\} \]
of $ \Sp_{2n}(\R) $. The assignment $ u\mapsto k_u $ defines a Lie group isomorphism $ \U(n)\cong K $. Every $ g\in\Sp_{2n}(\R) $ has the following factorization in $ \Sp_{2n}(\R) $:
\begin{equation}\label{eq:001}
	g=\begin{pmatrix}I_n&x\\0&I_n\end{pmatrix}\begin{pmatrix}y^{\frac12}&0\\0&y^{-\frac12}\end{pmatrix}k,
\end{equation}
where $ x+iy=g.(iI_n) $ and $ k\in K $. 

Let us denote the first two factors on the right-hand side of \eqref{eq:001}, from left to right, by $ n_x $ and $ a_y $. We fix the following Haar measure on $ \Sp_{2n}(\R) $:
\[ \int_{\Sp_{2n}(\R)}f(g)\,dg=\int_{\calH_n}\int_Kf(n_xa_yk)\,dk\,d\mathsf v(x+iy),\qquad f\in C_c(\Sp_{2n}(\R)), \]
where  $ \mathsf v $ is the $ \Sp_{2n}(\R) $-invariant Radon measure on $ \calH_n $ given by 
\[ d\mathsf v(x+iy)=\det y^{-n-1}\prod_{1\leq r\leq s\leq n}dx_{r,s}\,dy_{r,s}, \]
while $ dk $ is the Haar measure on $ K $ such that $ \int_Kdk=1 $. Given a discrete subgroup $ \Gamma $ of $ \Sp_{2n}(\R) $, there exists a unique Radon measure on $ \Gamma\backslash\Sp_{2n}(\R) $ such that
\[ \int_{\Gamma\backslash\Sp_{2n}(\R)}\sum_{\gamma\in\Gamma}f(\gamma g)\,dg=\int_{\Sp_{2n}(\R)}f(g)\,dg,\qquad f\in C_c(\Sp_{2n}(\R)). \]
With respect to this measure, for any finite-dimensional complex Hilbert space $ V $ and $ p\in\R_{\geq1} $, we define in a standard way the Banach space $ L^p(\Gamma\backslash\Sp_{2n}(\R),V) $ of (equivalence classes of) $ p $-integrable functions $ \Gamma\backslash\Sp_{2n}(\R)\to V $.
For all $ f\in C_c(\Gamma\backslash\Sp_{2n}(\R)) $, we have
\begin{equation}\label{eq:025}
	\int_{\Gamma\backslash\Sp_{2n}(\R)}f(g)\,dg=\frac1{\abs{\Gamma\cap\left\{\pm I_{2n}\right\}}}\int_{\Gamma\backslash\calH_n}\int_Kf(n_xa_y k)\,dk\,d\mathsf v(x+iy).
\end{equation}

Let $ \ell=\frac1{2i}\begin{pmatrix}I_n&-iI_n\\I_n&iI_n\end{pmatrix}\in\GL_{2n}(\C) $. The assignments
\[ z\mapsto\ell.z=(z-iI_n)(z+iI_n)^{-1}\qquad\text{and}\qquad w\mapsto\ell^{-1}.w=i(I_n+w)(I_n-w)^{-1} \]
define mutually inverse biholomorphisms between $ \calH_n $ and the bounded symmetric domain
\[ \calD_n=\left\{w\in M_n(\C):w^\top=w\text{ and }I_n-w^*w>0\right\}. \]
For every $ f\in C_c(\calH_n) $, we have
\[ \int_{\calH_n}f(z)\,d\mathsf v(z)=\int_{\calD_n}f(\ell^{-1}.w)d\mathsf v_\calD(w), \]
where, writing $ w=u+iv $ with $ u,v\in M_n(\R) $, $ \mathsf v_\calD $ is the ($ \ell\,\Sp_{2n}(\R)\ell^{-1} $)-invariant Radon measure on $ \calD_n $ given by
\begin{equation}\label{eq:007}
	 d\mathsf v_\calD(w)=2^{n(n+1)}\det(I_n-w^*w)^{-n-1}\prod_{1\leq r\leq s\leq n}du_{r,s}\,dv_{r,s}. 
\end{equation}

For $ r,s\in\left\{1,\ldots,n\right\} $, let $ E_{r,s}\in M_n(\C) $ be the matrix unit with $ (r,s) $th entry $ 1 $, and let $ T_r=-i\begin{pmatrix}&E_{r,r}\\-E_{r,r}\end{pmatrix}\in\mathfrak{gl}_{2n}(\C) $. The abelian Lie algebra
$ \mathfrak h=\sum_{r=1}^ni\R T_r \subseteq\mathfrak{gl}_{2n}(\R) $
 is a Cartan subalgebra of both $ \mathfrak k=\Lie(K) $ and $ \mathfrak g=\Lie(\Sp_{2n}(\R)) $. For every $ r $, we define a linear functional $ e_r\in\left(\mathfrak h_\C\right)^* $ by requiring that $ e_r(T_s)=\delta_{r,s} $ for all $ s $. Moreover, we introduce a strict total order $ \succ $ on $ \sum_{r=1}^n\R e_r $ by requiring that $ \sum_{r=1}^na_re_r\succ\sum_{r=1}^nb_re_r $ if there exists $ s\in\left\{1,\ldots,n\right\} $ such that
 \[ a_r=b_r\text{ for }r<s\qquad\text{and}\qquad a_s>b_s. \]      
With respect to $ \succ $, the set $ \Delta_K^+ $ of positive compact roots and the set $ \Delta_n^+ $ of positive non-compact roots in the root system $ \Delta=\Delta(\mathfrak h_\C:\mathfrak g_\C) $ are given by
\[ \Delta_K^+=\left\{e_r-e_s:1\leq r<s\leq n\right\}\quad\text{and}\quad \Delta_n^+=\left\{e_r+e_s:1\leq r\leq s\leq n\right\} \]
(cf.\ \cite[Ch.\,IX, \S7]{knapp86}). Denoting by $ \mathfrak g_\alpha $ the root subspace of $ \mathfrak g_\C $ corresponding to a root $ \alpha\in \Delta $, we define
\[ \mathfrak k_\C^+=\sum_{\alpha\in\Delta_K^+}\mathfrak g_\alpha,\quad\mathfrak p_\C^+=\sum_{\alpha\in\Delta_n^+}\mathfrak g_\alpha,\quad\mathfrak k_\C^-=\sum_{\alpha\in\Delta_K^+}\mathfrak g_{-\alpha},\quad\text{and}\quad \mathfrak p_\C^-=\sum_{\alpha\in\Delta_n^+}\mathfrak g_{-\alpha}. \]
                                                                                                                                                                                                                                                                                                             
\subsection{Irreducible representations of $ K $}\label{subsec:008}
We recall that, given an irreducible (continuous) representation $ \tau $ of $ K $ on a finite-dimensional complex vector space $ V $, a vector $ v\in V\setminus\left\{0\right\} $ is said to be a weight vector of weight $ \lambda\in\sum_{r=1}^n\Z e_r $  for $ \tau $ if $ \tau(H)v=\lambda(H)v $ for all $ H\in\mathfrak h_\C $. The weight vectors for $ \tau $ span $ V $, and those corresponding to the highest  (with respect to $ \succ $) weight $ \omega $ of $ \tau $ span the one-dimensional subspace $ V^{\mathfrak k_\C^+}=\left\{v\in V:\tau(\mathfrak k_\C^+)v=0 \right\} $. The equivalence classes of irreducible representations $ (\tau,V) $ of $ K $ are bijectively parametrized by their highest  weights $ \omega=\sum_{r=1}^n\omega_re_r $, where $ \omega_1,\ldots,\omega_n\in\Z $ and $ \omega_1\geq\ldots\geq\omega_n $. (See, e.g., \cite[(4.4)]{knapp86} and \cite[Thm.\,4.28]{knapp86}.)

\subsection{Irreducible polynomial representations of $ \GL_n(\C) $}\label{subsec:012}
Given an irreducible polynomial representation  $ \rho $  of $ \GL_n(\C) $ on a finite-dimensional complex vector space $ V $, we will always, applying \cite[Prop.\,1.6]{knapp86}, assume that $ V $ is equipped with a Hermitian inner product  $ \scal\spacedcdot\spacedcdot $  with respect to which the restriction $ \rho\big|_{\U(n)} $ is unitary and denote by $ \norm{\spacedcdot} $ the norm on $ V $. By \cite[\S2]{godement57_5}, we have the following lemma.
\begin{lemma}\label{lem:005}
	Let $ g\in\GL_n(\C) $. Then, we have 
	\begin{equation}\label{eq:002}
		\scal{\rho(g)v_1}{v_2}=\scal{v_1}{\rho\left(g^*\right)v_2},\qquad v_1,v_2\in V.
	\end{equation}
	Moreover, if the matrix $ g $ is positive definite, then so is the operator $ \rho(g) $, and we have $ \rho\big(g^{\frac12}\big)=\rho(g)^{\frac12} $.
\end{lemma}

Following \cite[\S8.2]{fulton97}, we say that a vector $ v\in V\setminus\left\{0\right\} $ is a weight vector of weight $ \lambda=(\lambda_1,\ldots,\lambda_n)\in\Z^n $ for $ \rho $  if 
\[ \rho(h)v=\left(\prod_{r=1}^nh_r^{\lambda_r}\right)v,\qquad h=\mathrm{diag}(h_1,\ldots,h_n)\in\mathrm{Diag}_n\big(\C^\times\big). \]
There exists an (up to scalar multiples unique) weight vector $ v^{top} $ for $ \rho $ such that $ \rho(B)v^{top}\subseteq\C v^{top} $, where $ B $ denotes the subgroup of upper-triangular matrices in $ \GL_n(\C) $. We say that $ v^{top} $ is a highest weight vector for $ \rho $ and call its weight the highest weight of $ \rho $. 
The equivalence classes of irreducible polynomial representations $ (\rho,V) $ of $ \GL_n(\C) $ are bijectively parametrized by their highest weights $ \omega=(\omega_1,\ldots,\omega_n)\in\Z^n $ such that $ \omega_1\geq\ldots\geq\omega_n\geq0 $ \cite[\S8.2, Thm.\,2(1)]{fulton97}. Moreover, by  \cite[\S8.1, Ex.\ 3]{fulton97} we have the following lemma.

\begin{lemma}\label{lem:004}
	Let $ \omega $ be the highest weight of an irreducible polynomial representation $ (\rho,V) $ of $ \GL_n(\C) $. Then, there exists a basis $ (e_T)_{T} $ of the vector space $ V $, indexed by the Young tableaux $ T $ of shape $ \omega $ with entries in $ \left\{1,2,\ldots,n\right\} $, with the following property: for each $ T $, $ e_T $ is a weight vector for $ \rho $ of weight $ \lambda=(\lambda_1,\ldots,\lambda_n)\in\left(\Z_{\geq\omega_n}\right)^n $, where $ \lambda_r $ is the number of times the integer $ r $ occurs in $ T $. 
\end{lemma}

	In the following lemma, we associate with a $ \GL_n(\C) $-representation $ (\rho,V) $ as above two $ K $-representations, $ \rho_K $ and $ \sigma_K $.

\begin{lemma}\label{lem:012}
	Let $ (\rho,V) $ be an irreducible polynomial representation of $ \GL_n(\C) $, and let $ \omega $ denote its highest weight. 
	\begin{enumerate}[label=\textup{(\roman*)}]
		\item\label{lem:012:1} The representation $ (\rho_K, V) $ of $ K $ defined by
		\[ \rho_K(k_u)=\rho(u),\qquad u\in\U(n), \]
		is an irreducible unitary representation of $ K $ of highest weight $ \sum_{r=1}^n\omega_re_r $.
		\item\label{lem:012:2} The representation $ (\sigma_K,V) $ of $ K $ defined by
		\begin{equation}\label{eq:024}
			\sigma_K(k_u)=\rho(\overline u),\qquad u\in\U(n),
		\end{equation}
		 is an irreducible unitary representation of $ K $ of highest weight $ -\sum_{r=1}^n\omega_{n+1-r}e_r $ and is equivalent to the contragredient representation $ (\rho_K^*,V^*) $.
		 \item\label{lem:012:3} If $ v^{top}\in V $ is a highest weight vector for $ \rho $, then $ v^{top} $ is also a highest weight vector for $ \rho_K $, and $ \left(v^{top}\right)^*=\scal\spacedcdot{v^{top}}\in V^* $ is a highest weight vector for $ \sigma_K^* $. 
	\end{enumerate}
\end{lemma}

\begin{proof}
	The representations $ \rho_K $, $ \sigma_K $, $ \rho_K^* $, and $ \sigma_K^* $ are irreducible and unitary since $ \rho\big|_{\mathrm U(n)} $ is. 
	Moreover, one shows by direct computation that if $ v\in V $ is a weight vector for $ \rho $  of weight $ \lambda\in\big(\Z_{\geq\omega_n}\big)^n $, then $ v $ is a weight vector for $ \rho_K $  of weight $ \sum_{r=1}^n\lambda_re_r $  and a weight vector for $ \sigma_K $ of weight $ -\sum_{r=1}^n\lambda_re_r $, and $ v^*=\scal\spacedcdot v\in V^* $ is a weight vector for $ \rho_K^* $ of weight $ -\sum_{r=1}^n\lambda_re_r $ and a weight vector for $ \sigma_K^* $ of weight $ \sum_{r=1}^n\lambda_re_r $. This implies the lemma by the description of weights of $ \rho $ given in Lem.\,\ref{lem:004} and the characterization of equivalence classes of irreducible representations of $ K $ by their highest weights (see \S\ref{subsec:008}).
\end{proof}

\section{Holomorphic and antiholomorphic discrete series of $ \Sp_{2n}(\R) $}

Let $ (\rho,V) $ be an irreducible polynomial representation of $ \GL_n(\C) $ of highest weight $ \omega=(\omega_1,\ldots,\omega_n)\in\Z^n $, where $ \omega_1\geq\ldots\geq\omega_n\geq0 $. As always, we assume that the restriction $ \rho\big|_{\mathrm U(n)} $ is unitary (see \S\ref{subsec:012}).

The group $ \Sp_{2n}(\R) $ acts on the right on the space $ V^{\calH_n} $ of functions $ \calH_n\to V $ by
\[ \left(f\big|_\rho g\right)(z)=\rho(Cz+D)^{-1}f(g.z) \]
for all $ f\in V^{\calH_n} $, $ g=\begin{pmatrix}A&B\\C&D\end{pmatrix}\in\Sp_{2n}(\R) $, and $ z\in\calH_n $.
It follows from \eqref{eq:019} and \eqref{eq:002} that
\begin{equation}\label{eq:020}
	\norm{\rho\Big(y^{\frac12}\Big)\left(f\big|_\rho g\right)(z)}=\norm{\rho\left(\Im(g.z)^{\frac12}\right)f(g.z)}
\end{equation}
for all $ f\in V^{\calH_n} $, $ g\in\Sp_{2n}(\R) $, and $ z=x+iy\in\calH_n $ (see, e.g., \cite[\S3]{weissauer83}).

Suppose that $ \omega_n>n $. Let $ \calH^2(\rho) $ be the Hilbert space of holomorphic functions $ f:\calH_n\to V $ such that
\[ \int_{\calH_n}\norm{\rho\Big(y^{\frac12}\Big)f(z)}^2\,d\mathsf v(z)<\infty, \]
with the inner product
\[ \begin{aligned}
	\scal{f_1}{f_2}_{\calH^2(\rho)}&\overset{\phantom{\eqref{eq:002}}}=\int_{\calH_n}\scal{\rho\Big(y^{\frac12}\Big)f_1(z)}{\rho\Big(y^{\frac12}\Big)f_2(z)}\,d\mathsf v(z)\\
	&\overset{\eqref{eq:002}}=\int_{\calH_n}\scal{\rho(y)f_1(z)}{f_2(z)}\,d\mathsf v(z).
\end{aligned} \]
Let $ (\pi_\rho,\calH^2(\rho)) $ be the unitary representation of $ \Sp_{2n}(\R) $ defined by 
\[ \pi_\rho(g)f=f\big|_\rho g^{-1},\qquad g\in\Sp_{2n}(\R),\ f\in\calH^2(\rho). \] 
The representations $ (\pi_\rho,\calH^2(\rho)) $ comprise the holomorphic discrete series of $ \Sp_{2n}(\R) $ (see, e.g., \cite[\S3]{gelbart73}).
 
The assignment $ f\mapsto f\big|_\rho\ell^{-1} $ defines a unitary equivalence of $ \Sp_{2n}(\R) $-representations $ (\pi_\rho,\mathcal H^2(\rho)) $ and $ (\pi_\rho^{\ell},\calD^2(\rho)) $, where $ \calD^2(\rho) $ is the Hilbert space of holomorphic functions $ f:\calD_n\to V $ such that
\begin{equation}\label{eq:004}
	\int_{\calD_n}\norm{\rho\left(I_n-w^*w\right)^{\frac12}f(w)}^2\,d\mathsf v_{\calD}(w)<\infty, 
\end{equation}
equipped with the inner product
\begin{equation}\label{eq:117}
	\scal{f_1}{f_2}_{\calD^2(\rho)}=\int_{\calD_n}\scal{\rho\left(I_n-w^*w\right)^{\frac12}f_1(w)}{\rho\left(I_n-w^*w\right)^{\frac12}f_2(w)}\,d\mathsf v_{\calD}(w),
\end{equation}
and
\[ \pi_\rho^\ell(g)f=f\big|_\rho \ell\,g^{-1}\ell^{-1},\qquad g\in\Sp_{2n}(\R),\ f\in\calD^2(\rho). \]

Let $ \C[X_{r,s}:1\leq r,s\leq n] $ denote the ring of polynomials with complex coefficients in the $ n^2 $ variables $ X_{r,s} $, where $ r,s\in\left\{1,2,\ldots,n\right\} $. Every $ \mu\in\C[X_{r,s}:1\leq r,s\leq n] $ can be evaluated at any matrix $ z\in M_n(\C) $ in an obvious way; in particular, for every $ v\in V $, we can define a function $ p_{\mu,v}:\calD_n\to V $,
\begin{equation}\label{eq:003}
	 p_{\mu,v}(w)=\mu(w)v. 
\end{equation}
For each $ d\in\Z_{\geq0} $, let $ \calD^2(\rho)_d $ denote the linear span of functions $ p_{\mu,v} $, where $ \mu $ runs over the degree $ d $ homogeneous polynomials in $ \C[X_{r,s}:1\leq r,s\leq n] $ and $ v $ runs through $ V $. 

\begin{lemma}\label{lem:010}
	\begin{enumerate}[label=\textup{(\roman*)}]
		\item\label{lem:010:1} The $ (\mathfrak g,K) $-module $ \calD^2(\rho)_K $ of $ K $-finite vectors in $ \calD^2(\rho) $ has the following decomposition into a direct sum of mutually orthogonal finite-dimensional $ K $-invariant subspaces:
		\begin{equation}\label{eq:009}
			\calD^2(\rho)_K=\bigoplus_{d=0}^\infty\calD^2(\rho)_d. 
		\end{equation}
		\item\label{lem:010:2} Denoting $ \calD^2(\rho)_{-1}=0 $, we have
		\[ \pi_\rho^\ell\left(\mathfrak p_\C^{\pm}\right)\calD^2(\rho)_r\subseteq\calD^2(\rho)_{r\mp1},\qquad r\in\Z_{\geq0}. \] 
		\item\label{lem:010:3} The assignment $ v\mapsto p_{1,v} $ defines an equivalence of $ K $-representations $ \left(\sigma_K,V\right) $ and $ \calD^2(\rho)_0 $.
	\end{enumerate}
	
\end{lemma}

\begin{proof}
	This is well-known (cf.\ \cite[Ch.\,VIII, \S2]{knapp86}). A proof of \ref{lem:010:1} in the case when $ n=1 $ can be found in \cite[Ch.\,IX, \S2, Thm.\,3]{lang85} and  generalized easily to the case when $ n>1 $ using the convergence properties of power series of holomorphic functions $ \calD_n\to\C $ described in \cite[Ch.\,III, \S6, (6)]{klingen90}.
	For example, to prove that the subspaces $ \calD^2(\rho)_d $ are $ K $-invariant, it suffices to note that for all $ \mu\in\C[X_{r,s}:1\leq r,s,\leq n] $, $ v\in V $, and $ u\in\mathrm U(n) $, we have
	\[ \left(\pi_\rho^\ell(k_u)p_{\mu,v}\right)(w)=\mu\left(u^{-1}wu^{-\top}\right)\rho\left(u^{-\top}\right)v,\qquad w\in\calD_n. \]
	On the other hand, the mutual orthogonality of subspaces $ \calD^2(\rho)_d $ is easily proved using the invariance of the measure $ \mathsf v_\calD $ under $ \ell\,k_{e^{it}I_n}\ell^{-1} $ for $ t\in\R $. The claims \ref{lem:010:2}--\ref{lem:010:3} are proved by direct computation. We leave the details to the reader. 
\end{proof}
	
	By \eqref{eq:009}, we have
	\begin{equation}\label{eq:061}
		\calH^2(\rho)_K=\mathrm{span}_\C\left\{f_{\mu,v}:\mu\in\C[X_{r,s}:1\leq r,s\leq n],\ v\in V\right\},
	\end{equation}
	where the functions $ f_{\mu,v}:\calH_n\to V $ are given by
	\begin{equation}\label{eq:048}
		f_{\mu,v}(z)=\left(p_{\mu,v}\big|_\rho\ell\right)(z)=\mu\left((z-iI_n)(z+iI_n)^{-1}\right)\rho\left(\frac1{2i}(z+iI_n)\right)^{-1}v. 
	\end{equation}
	
	In Sec.\,\ref{sec:042}, we will need the following characterization of representations $ \pi_\rho $ and $ \pi_\rho^* $.
	
	\begin{lemma}\label{lem:011}
		Let $ \rho $ be an irreducible polynomial representation of $ \GL_n(\C) $ of highest weight $ \omega=(\omega_1,\ldots,\omega_n)\in\Z^n $, where $ \omega_1\geq\ldots\geq\omega_n>n $. Then, we have the following:
		\begin{enumerate}[label=\textup{(\roman*)}]
			\item\label{lem:011:1} The representation $ \pi_\rho $ is the up to unitary equivalence unique irreducible unitary representation $ (\pi,H) $ of $ \Sp_{2n}(\R)  $ containing a $ K $-invariant subspace $ U $ that is $ K $-equivalent to $ \rho_K^* $ and annihilated by $ \mathfrak p_\C^+ $.
			\item\label{lem:011:2} The representation $ \pi_\rho^* $ is the up to unitary equivalence unique irreducible unitary representation $ (\pi,H) $ of $ \Sp_{2n}(\R)  $ containing a $ K $-invariant subspace $ U $ that is $ K $-equivalent to $ \rho_K $ and annihilated by $ \mathfrak p_\C^- $.
		\end{enumerate}
	\end{lemma}

	\begin{proof}
		In the case when $ \rho $ is a character, the lemma follows from \cite[Lem.\,3.1]{zunar23}.
		
		In the general case, we note that the representation $ \pi_\rho $ has the properties stated in \ref{lem:011:1} since the subspace $ \calD^2(\rho)_0 $ of $ (\pi_\rho^\ell,\calD^2(\rho))\cong(\pi_\rho,\calH^2(\rho)) $ is annihilated by $ \mathfrak p_\C^+ $ by Lem.\,\ref{lem:010}\ref{lem:010:2} and is $ K $-equivalent to $ \sigma_K\cong \rho_K^* $ by Lem.\,\ref{lem:010}\ref{lem:010:3} and Lem.\,\ref{lem:012}\ref{lem:012:2}. 
		
		On the other hand, given an $ \Sp_{2n}(\R) $-representation $ (\pi,H) $ and a subspace $ U\subseteq H $ as in \ref{lem:011:1}, let us fix a weight vector $ v^{top}\in U $ of (the highest) weight $ -\sum_{r=1}^n\omega_{n+1-r}e_r $ (see Lem.\,\ref{lem:012}\ref{lem:012:2}). The vector $ v^{top} $ is annihilated by both $ \mathfrak k_\C^+ $ (by \S\ref{subsec:008}) and $\mathfrak p_\C^+ $ (by the assumption on $ U $) and generates the irreducible (by \cite[Thm.\,3.4.11]{wallach88} and the connectedness of $ K $) $ \mathfrak g_\C $-module $ H_K $ of $ K $-finite vectors in $ H $. Thus, $ H_K $ is, in the sense of \cite[\S20.2]{humphreys72}, an irreducible standard cyclic $ \mathfrak g_\C $-module of highest weight $ -\sum_{r=1}^n\omega_{n+1-r}e_r $. By \cite[\S20.3, Thm.\,A]{humphreys72}, this property uniquely determines the $ \mathfrak g_\C $-module isomorphism class of $ H_K $ and hence, by the connectedness of $ K $, the $ (\mathfrak g,K) $-module isomorphism class of $ H_K $ or equivalently, by \cite[Thm.\,3.4.11]{wallach88}, the unitary $ \Sp_{2n}(\R) $-equivalence class of $ (\pi,H) $. 
		
		This proves \ref{lem:011:1}, which implies \ref{lem:011:2} by duality since $ \mathfrak p_\C^-=\overline{\mathfrak p_\C^+} $.
	\end{proof}

	With reference to \cite[Thm.\,9.20]{knapp86}, $ \pi_\rho $ is a discrete series representation of $ \Sp_{2n}(\R) $ with Harish-Chandra parameter $ -\sum_{r=1}^n(\omega_r-r)e_r $, Blattner parameter $ -\sum_{r=1}^n\omega_re_r $, and lowest $ K $-type $ \rho_K^* $. On the other hand, $ \pi_\rho^* $ is a discrete series representation of $ \Sp_{2n}(\R) $ with Harish-Chandra parameter $ \sum_{r=1}^n(\omega_r-r)e_r $, Blattner parameter $ \sum_{r=1}^n\omega_re_r $, and lowest $ K $-type $ \rho_K $. In particular, we have the following lemma.
	
	\begin{lemma}\label{lem:037}
		The $ K $-type $ \rho_K^* $ (resp., $ \rho_K $) occurs in $ \pi_\rho $ (resp., $ \pi_\rho^* $) with multiplicity one.
	\end{lemma}
	
	Let us note that by Lem.\ \ref{lem:012}\ref{lem:012:2}, Lem.\,\ref{lem:010}\ref{lem:010:3}, and Lem.\,\ref{lem:037}, we have the following lemma.
	\begin{lemma}\label{lem:116}
		\begin{enumerate}[label=\textup{(\roman*)}]
			\item\label{lem:116:1} The assignment $ v\mapsto f_{1,v} $ defines a $ K $-equivalence of $ \left(\sigma_K,V\right) $ with 
			\[ \calH^2(\rho)_{[\sigma_K]}=\left\{f_{1,v}:v\in V\right\}. \]
			\item\label{lem:116:2} The assignment $ v^*\mapsto f_{1,v}^* $ defines a $ K $-equivalence of $ \left(\sigma_K^*,V^*\right) $ with 
			\begin{equation}\label{eq:112}
				\left(\calH^2(\rho)^*\right)_{[\rho_K]}=\left\{f_{1,v}^*:v\in V\right\}. 
			\end{equation}
		\end{enumerate}
	\end{lemma}
	
 	An irreducible unitary representation $ (\pi,H) $ of $ \Sp_{2n}(\R) $ is said to be integrable if for every choice of $ K $-finite vectors $ h,h'\in H_K $, the matrix coefficient $ c_{h,h'}:\Sp_{2n}(\R)\to\C $,
	\begin{equation}\label{eq:039}
		c_{h,h'}(g)=\scal{\pi(g)h}{h'}_H,
	\end{equation}
 	belongs to $ L^1(\Sp_{2n}(\R)) $. By applying the criterion for integrability of discrete series representations given in \cite{hecht_schmid76} (see also \cite{milicic77}), we obtain the following lemma.
 	
 	\begin{lemma}\label{lem:041}
 		The representation $ \pi_\rho $ (resp., $ \pi_\rho^* $) is integrable if and only if $ \omega_n>2n $.
 	\end{lemma}

	\section{Vector-valued Siegel cusp forms}\label{sec:104}
 	
 	From now until the end of the paper, let:
 	\begin{enumerate}[label=\textup{(\arabic*)}]
 		\item $ \Gamma $ be a subgroup of $ \Sp_{2n}(\R) $ that is commensurable with the Siegel modular group $ \Sp_{2n}(\Z) $, i.e., is such that $ \Gamma\cap\Sp_{2n}(\Z) $ is of finite index in both $ \Gamma $ and $ \Sp_{2n}(\Z) $. 
 		\item $ (\rho,V) $ be an irreducible polynomial representation of $ \GL_n(\C) $ of highest weight $ \omega=(\omega_1,\ldots,\omega_n)\in\Z^n $, where $ \omega_1\geq\ldots\geq\omega_n\geq0 $ (see \S\ref{subsec:012}).
 	\end{enumerate}
 	
 	There exists an integer $ N\in\Z_{>0} $ such that 
 	\[ \left\{\begin{pmatrix}I_n&NS\\&I_n\end{pmatrix}:S\in M_n(\Z)\text{ and } S^\top=S\right\}\subseteq\bigcap_{\delta\in\Sp_{2n}(\Z)}\delta\Gamma\delta^{-1}. \]
 	We will denote the smallest such integer $ N $ by $ N_\Gamma $ and call it the level of $ \Gamma $. An important example of a group $ \Gamma $ of level $ N\in\Z_{>0} $ is the principal level $ N $ congruence subgroup
 	\[ \Gamma_n(N)=\left\{\gamma\in\Sp_{2n}(\Z):\gamma\cong I_{2n}\pmod N\right\} \]
 	of $ \Sp_{2n}(\Z) $.
 	 
	Let $ f:\calH_n\to V $ be a holomorphic function such that
	\[ f\big|_\rho\gamma=f,\qquad \gamma\in\Gamma. \]
	It follows from \cite[Ch.\,I, \S0.12]{freitag83} that for every $ \delta\in\Sp_{2n}(\Z) $, we have an absolutely and locally uniformly convergent Fourier expansion
	\begin{equation}\label{eq:015}
		\left(f\big|_\rho\delta\right)(z)=\sum_{T}e^{\frac{2\pi i}{N_\Gamma}\tr(Tz)}a_T(\delta),\qquad z\in\calH_n,
	\end{equation}
	where the sum goes over the symmetric matrices $ T=(t_{r,s})\in M_n(\Q) $ that are half-integral, i.e., satisfy $ t_{r,r}\in\Z $ and $ 2t_{r,s}\in\Z $ for all $ r,s\in\left\{1,\ldots,n\right\} $. The vectors $ a_T(\delta)\in V $ are given, using the notation $ z=x+iy $ for elements of $ \calH_n $ and writing $ m=\frac{n(n+1)}2 $, by the formula
	\begin{equation}\label{eq:017}
		a_T(\delta)=N_\Gamma^{-m}\int_{[0,N_\Gamma]^m}e^{-\frac{2\pi i}{N_\Gamma}\tr(Tz)}\left(f\big|_\rho \delta\right)(z)\,\prod_{1\leq r\leq s\leq n}dx_{r,s}
	\end{equation}
 	for every positive definite matrix $ y\in\GL_n(\R) $. One shows as in \textup{\cite[Ch.\,2, proof of Thm.\,3.13]{andrianov_zhuravlev95}} that
 	\begin{equation}\label{eq:022}
 		a_T(\delta)=\rho\left(U^{-\top}\right)a_{U^\top TU}\left(\delta \begin{pmatrix}U&\\&U^{-\top}\end{pmatrix}\right),\qquad U\in\GL_n(\Z).
 	\end{equation}

	A Siegel cusp form of weight $ \rho $ for $ \Gamma $ is a holomorphic function $ f:\calH_n\to V $ with the following properties:
	 \begin{enumerate}[label=\textrm{(S\arabic*)}]
	 	\item\label{def:016:1} $ f\big|_\rho\gamma=f $ for all $ \gamma\in\Gamma $.
	 	\item\label{def:016:2} For every $ \delta\in\Sp_{2n}(\Z) $, the coefficients $ a_T(\delta) $ in the Fourier expansion \eqref{eq:015} of $ f\big|_\rho\delta $ satisfy the implication
	 		\[ a_T(\delta)\neq0\quad\Rightarrow\quad T>0.  \]
	 \end{enumerate} 
  	We will denote by $ S_\rho(\Gamma) $ the finite-dimensional complex Hilbert space of Siegel cusp forms of weight $ \rho $ for $ \Gamma $, equipped with the Petersson inner product
  	\begin{equation}\label{eq:026}
  		\scal{f_1}{f_2}_{S_\rho(\Gamma)}=\frac1{\abs{\Gamma\cap\left\{\pm I_{2n}\right\}}}\int_{\Gamma\backslash\calH_n}\scal{\rho\Big(y^{\frac12}\Big)f_1(z)}{\rho\Big(y^{\frac12}\Big)f_2(z)}\,d\mathsf v(z)
  	\end{equation}
  	(cf.\ \cite[Ch.\,I, Def.\,4.6]{freitag91} and \cite[\S1 and \S3]{weissauer83}).

  	It is well-known that $ S_\rho(\Gamma)=0 $ if $ \omega_n=0 $ (for $ n=1 $, see \cite[Thm.\,2.5.2]{miyake06}; for $ n>1 $, see\cite[Ch.\,I, Prop.\,4.7 and Cor.\ of Thm.\,1.4]{freitag91}). If $ \omega_n>0 $, we have the following characterization of $ S_\rho(\Gamma) $, whose proof in the case when $ \Gamma=\Sp_{2n}(\Z) $ can be found in \cite[\S3, Cor.\,1 and (3.9)--(3.11)]{godement57_7}.
  	
  	\begin{lemma}\label{lem:026}
 		Suppose that $ \omega_n>0 $. Let $ f:\calH_n\to V $ be a holomorphic function satisfying \ref{def:016:1}. Then, $ f $ belongs to $ S_\rho(\Gamma) $ if and only if 
 		\begin{equation}\label{eq:016}
 			\sup_{z\in\calH_n}\norm{\rho\Big(y^{\frac12}\Big)f(z)}<\infty.
 		\end{equation}
 	\end{lemma}
 
 	\begin{proof}
		In the case when $ n=1 $, the lemma follows from \cite[Thm.\,2.1.5]{miyake06}. Suppose that $ n>1 $. In the case when $ \dim_\C V=1 $ and $ \Gamma $ is a congruence subgroup of $ \Sp_{2n}(\Z) $, the left-to-right implication is proved in \cite[Ch.\,2, proof of Thm.\,3.13(2)]{andrianov_zhuravlev95}, and the proof generalizes straightforwardly to our more general case using \cite[Ch.\,I, Cor.\ of Thm\,1.4]{freitag91}. 
 			
 			Conversely, assuming that \eqref{eq:016} holds and that $ n>1 $, let us show that $ f $ satisfies \ref{def:016:2}.
 			Let $ \delta\in\Sp_{2n}(\Z) $. Since $ n>1 $, by the Koecher principle, we have that $ a_T(\delta)=0 $ if $ T\not\geq0 $ (see, e.g., \cite[Ch.\,I, \S4]{freitag91}). If $ T\geq0 $ and $ \det T=0 $, then by \textup{\cite[Ch.\,2, Lem.\,3.14]{andrianov_zhuravlev95}} there exists $ U\in\SL_n(\Z) $ such that $ U^\top TU $ is of the form
 			\begin{equation}\label{eq:023}
 				\begin{pmatrix}S&0\\0&0\end{pmatrix} 
 			\end{equation}
 			for some positive semidefinite, half-integral matrix $ S\in M_{n-1}(\Q) $,
 			hence by \eqref{eq:022} it suffices to prove that $ a_T(\delta)=0 $ for $ T\geq0 $ of the form \eqref{eq:023}. Given such $ T $ and any positive definite matrix $ y_0\in \GL_n(\R) $, we have the estimate
 			\begin{equation}\label{eq:021}
 				\begin{aligned}
 				\norm{a_T(\delta)}
 				&\overset{\eqref{eq:017}}\leq \sup_{\substack{z\in\calH_n\\\Im(z)=y_0}}\abs{e^{-\frac{2\pi i}{N_\Gamma}\tr(Tz)}}\,\norm{\left(f\big|_\rho\delta\right)(z)}\\
 				&\overset{\phantom{\eqref{eq:020}}}\leq e^{\frac{2\pi}{N_\Gamma}\tr(Ty_0)}\,\norm{\rho\Big(y_0^{-\frac12}\Big)}_{\mathrm{op}}\,\sup_{\substack{z\in\calH_n\\\Im(z)=y_0}}\norm{\rho\left(y_0^{\frac12}\right)\left(f\big|_\rho\delta\right)(z)}\\
 				&\overset{\eqref{eq:020}}\leq e^{\frac{2\pi}{N_\Gamma}\tr(Ty_0)}\,\norm{\rho\Big(y_0^{-\frac12}\Big)}_{\mathrm{op}}\,\sup_{z\in\calH_n}\norm{\rho\Big(y^{\frac12}\Big)f(z)},
 			\end{aligned}
 			\end{equation}
 			where $ \norm{\spacedcdot}_{\mathrm{op}} $ denotes the operator norm. Let
 			\[ y_0=\begin{pmatrix}I_{n-1}&0\\0&t^{-2}\end{pmatrix},\qquad\text{i.e.,}\qquad y_0^{-\frac12}=\begin{pmatrix}I_{n-1}&0\\0&t\end{pmatrix}, \] 
 			for some $ t\in\left]0,1\right] $. We have $ \mathrm{tr}(Ty_0)=\mathrm{tr}(S) $, and the operator $ \rho\Big(y_0^{-\frac12}\Big) $ is positive definite with eigenvalues $ t^{\lambda_n} $, where $ \lambda $ runs over the weights of $ \rho $, hence
 			\[ \norm{\rho\Big(y_0^{-\frac12}\Big)}_{\mathrm{op}}=\max_{\lambda\text{ a weight of }\rho}t^{\lambda_n}=t^{\omega_n} \]
 			by Lem.\,\ref{lem:004}. Thus, the estimate \eqref{eq:021} implies that
 			\[ \norm{a_T(\delta)}
 			\leq e^{\frac{2\pi}{N_\Gamma}\tr(S)}\,t^{\omega_n}\,\sup_{z\in\calH_n}\norm{\rho\Big(y^{\frac12}\Big)f(z)}. \]
 			Since $ \omega_n>0 $ and \eqref{eq:016} holds, the right-hand side tends to $ 0 $ as $ t\searrow0 $, hence $ a_T(\delta)=0 $. 
 	\end{proof}
 
	Let us recall the classical lift of a function $ f:\calH_n\to V $ to a function $ F_f:\Sp_{2n}(\R)\to V $ given by
	\begin{equation}\label{eq:047}
		 F_f(g)=\left(f\big|_\rho g\right)(iI_n)=\sigma_K(k)^{-1}\rho\Big(y^{\frac12}\Big)f(x+iy) 
	\end{equation}
	for all $  g=n_xa_yk\in\Sp_{2n}(\R) $ (see \eqref{eq:001}), where $ \sigma_K $ is defined by \eqref{eq:024}.
	
	\begin{lemma}\label{lem:021}
		Let $ f:\calH_n\to V $. Then, we have the following:
		\begin{enumerate}[label=\textup{(\roman*)}]
			\item\label{lem:021:1} For every $ g\in\Sp_{2n}(\R) $, the following equivalence holds:
			\[ f\big|_\rho g=f\quad\Leftrightarrow\quad F_f(g\spacedcdot)=F_f. \]
			\item\label{lem:021:2} The function $ f $ is holomorphic if and only if the function $ F_f $ is smooth and annihilated by $ \mathfrak p_\C^- $, where $ \mathfrak p_\C^- $ acts by left-invariant differential operators.
			\item\label{lem:021:3} $ F_f(gk)=\sigma_K(k)^{-1}F_f(g) $ for all $ g\in\Sp_{2n}(\R) $ and $ k\in K $.
			\item\label{lem:021:4} $ \norm{F_f(g)}=\norm{\rho\Big(y^{\frac12}\Big)f(x+iy)} $ for all $ g=n_xa_yk\in\Sp_{2n}(\R) $.
		\end{enumerate}
	\end{lemma}
	
	\begin{proof}
		The claims \ref{lem:021:1}, \ref{lem:021:3}, and \ref{lem:021:4}  are elementary, and \ref{lem:021:2} is proved by generalizing \cite[proof of Lem.\,7]{asgari_schmidt01} in a straightforward way. We leave the details to the reader.
 	\end{proof}
	
	Let $ \mathcal A_{cusp}(\Gamma\backslash\Sp_{2n}(\R),\rho,V)^{\mathfrak p_\C^-} $ denote the space of smooth functions $ \varphi:\Sp_{2n}(\R)\to V $ with the following properties:
	\begin{enumerate}[label=\textup{(V\arabic*)}]
		\item\label{enum:022:1} $ \varphi(\gamma\spacedcdot)=\varphi $ for all $ \gamma\in\Gamma $.
		\item\label{enum:022:2} $ \mathfrak p_\C^-\varphi=0 $.
		\item\label{enum:022:3} $ \varphi(gk)=\sigma_K(k)^{-1}\varphi(g) $ for all $ g\in\Sp_{2n}(\R) $ and $ k\in K $.
		\item\label{enum:022:4} $ \varphi $ is cuspidal, i.e., for every proper $ \Q $-parabolic subgroup $ P $ of $ \Sp_{2n} $, we have
		\[ \int_{(\Gamma\,\cap\, U_P(\R))\backslash U_P(\R)}\varphi(ug)\,du=0,\qquad g\in\Sp_{2n}(\R), \]
		where $ U_P $ is the unipotent radical of $ P $.
		\item\label{enum:022:5} $ \varphi $ is bounded, i.e., $ \sup_{g\in\Sp_{2n}(\R)}\norm{\varphi(g)}<\infty $.
	\end{enumerate}

	The following lemma is a generalization of \cite[Prop.\,4.1]{zunar23} from the scalar-valued to the vector-valued case.
	
	\begin{lemma}\label{lem:058}
		Suppose that $ \omega_n>0 $. Then, the assignment $ f\mapsto F_f $ defines a unitary isomorphism 
		\[ \Phi_{\rho,\Gamma}:S_\rho(\Gamma)\to\mathcal A_{cusp}(\Gamma\backslash\Sp_{2n}(\R),\rho,V)^{\mathfrak p_\C^-},\] 
		where $ \mathcal A_{cusp}(\Gamma\backslash\Sp_{2n}(\R),\rho,V)^{\mathfrak p_\C^-} $ is regarded as a subspace of the Hilbert space $ L^2(\Gamma\backslash\Sp_{2n}(\R),V) $.
	\end{lemma}
	
	\begin{proof}
		Given $ f\in S_\rho(\Gamma) $, the function $ F_f:\Sp_{2n}(\R)\to V $ is obviously smooth, satisfies \ref{enum:022:1}--\ref{enum:022:3} and \ref{enum:022:5} by Lem.\,\ref{lem:021} and \eqref{eq:016}, and satisfies \ref{enum:022:4} by the argument given in \cite[proof of Lem.\,5]{asgari_schmidt01}. Moreover, $ \norm{F_f}_{L^2(\Gamma\backslash\Sp_{2n}(\R),V)}=\norm f_{S_\rho(\Gamma)} $ by \eqref{eq:025}, Lem.\,\ref{lem:021}\ref{lem:021:4}, and \eqref{eq:026}. Thus, $ \Phi_{\rho,\Gamma} $ is a well-defined linear isometry.
		
		On the other hand, given $ \varphi\in\mathcal A_{cusp}(\Gamma\backslash\Sp_{2n}(\R),\rho,V)^{\mathfrak p_\C^-} $, the function $ f:\calH_n\to V $, $ f(z)=\rho\Big(y^{-\frac12}\Big)\varphi(n_xa_y) $, obviously satisfies $ F_f=\varphi $. Moreover, $ f $ is holomorphic and satisfies \ref{def:016:1} by Lem.\,\ref{lem:021}\ref{lem:021:1}--\ref{lem:021:2}, satisfies \eqref{eq:016} by \ref{enum:022:5} and Lem.\,\ref{lem:021}\ref{lem:021:4}, and hence belongs to $ S_\rho(\Gamma) $ by Lem.\,\ref{lem:026}. Thus, $ \Phi_{\rho,\Gamma} $ is surjective.
	\end{proof}

	Given $ \varphi:\Sp_{2n}(\R)\to V $ and $ v\in V $, we define 
	\begin{equation}\label{eq:032}
		v^*=\scal\spacedcdot v\in V^*\qquad\text{and}\qquad v^*\varphi=\scal{\varphi(\spacedcdot)}v:\Sp_{2n}(\R)\to\C.
	\end{equation}
	In Sec.\ \ref{sec:042}, we will need the following lemma.
	
	\begin{lemma}\label{lem:029}
		Let $ \varphi:\Sp_{2n}(\R)\to V $ be such that \ref{enum:022:3} holds. Then, we have the following:
		\begin{enumerate}[label=\textup{(\roman*)}]
			\item\label{lem:029:1} If $ v^*\varphi $ vanishes identically for some $ v\in V\setminus\left\{0\right\} $, then $ \varphi $ vanishes identically.
			\item\label{lem:029:2} Suppose that $ \varphi $ does not vanish identically. Then, the assignment  
			\begin{equation}\label{eq:031}
				 v^*\mapsto v^*\varphi,\qquad v\in V,
			\end{equation}
			defines a $ K $-equivalence from $ \left(\sigma_K^*,V^*\right) $ to the right regular representation $ (R,U_\varphi) $ of $ K $ on $ U_\varphi=\left\{v^*\varphi:v\in V\right\}.$
		\end{enumerate}
	\end{lemma}

	\begin{proof}
		\ref{lem:029:1} If the function $ v^*\varphi $ vanishes identically for some $ v\in V\setminus\left\{0\right\} $, then so does, for every $ k\in K $, the right translate
		\begin{equation}\label{eq:030}
			 \left(v^*\varphi\right)(\spacedcdot k)
			 \overset{\eqref{eq:032}}=\scal{\varphi(\spacedcdot k)}v
			\overset{\ref{enum:022:3}}=\scal{\sigma_K(k)^{-1}\varphi(\spacedcdot)}v=\scal{\varphi(\spacedcdot)}{\sigma_K(k)v}.
		\end{equation}
		Since the set $ \sigma_K(K)v $ spans $ V $ by the irreducibility of $ \sigma_K $, it follows that $ \varphi $ vanishes identically.
		
		\ref{lem:029:2} The linear operator in question is injective by \ref{lem:029:1}, is obviously surjective, and is $ K $-equivariant since
		by \eqref{eq:030}, for all $ v\in V $ and $ k\in K $, the assignment \eqref{eq:031} takes $ \sigma_K^*(k)v^* $ to $ R(k)\left(v^*\varphi\right) $.
 	\end{proof}

	\section{Cuspidal automorphic forms on $ \Sp_{2n}(\R) $}\label{sec:042}
	
	Following \cite[\S1.3 and \S1.8]{borel_jacquet79}, we define the $ (\mathfrak g,K) $-module $ \mathcal A_{cusp}(\Gamma\backslash\Sp_{2n}(\R)) $ of cuspidal automorphic forms for $ \Gamma $ to consist of smooth functions $ \varphi:\Sp_{2n}(\R)\to\C $ with the following properties:
	\begin{enumerate}[label=\textup{(A\arabic*)}]
		\item \label{enum:027:1} $ \varphi(\gamma\spacedcdot)=\varphi $ for all $ \gamma\in\Gamma $.
		\item\label{enum:027:2} $ \varphi $ is $ \mathcal Z(\mathfrak g) $-finite, i.e., $ \dim_\C\mathcal Z(\mathfrak g)\varphi<\infty $.
		\item\label{enum:027:3} $ \varphi $ is right $ K $-finite, i.e., $ \dim_\C\mathrm{span}_\C\left\{\varphi(\spacedcdot k):k\in K\right\}<\infty $.
		\item\label{enum:027:4} $ \varphi $ is cuspidal.
		\item\label{enum:027:5} $ \varphi $ satisfies the following (under the above assumptions mutually equivalent) conditions: 
		\begin{enumerate}[label=\textup{(A5-\arabic*)}]
			\item $ \varphi $ is bounded.
			\item $ \varphi\in L^2(\Gamma\backslash\Sp_{2n}(\R)) $.
		\end{enumerate}
	\end{enumerate}

	Since $ \mathrm{Ad}(K)\mathfrak p_\C^-=\mathfrak p_\C^- $, the subspace 
	\[ \mathcal A_{cusp}(\Gamma\backslash\Sp_{2n}(\R))^{\mathfrak p_\C^-}=\left\{\varphi\in\mathcal A_{cusp}(\Gamma\backslash\Sp_{2n}(\R)):\mathfrak p_\C^-\varphi=0\right\} \] 
	of $ \mathcal A_{cusp}(\Gamma\backslash\Sp_{2n}(\R)) $ is $ K $-invariant. Let $ \mathcal A_{cusp}(\Gamma\backslash\Sp_{2n}(\R))^{\mathfrak p_\C^-}_{[\rho_K]} $ denote its $ \rho_K $-isotypic component, i.e., the algebraic sum of its $ K $-invariant subspaces that are equivalent to $ \rho_K $. In the following lemma, we use the notation \eqref{eq:032}.
	
	\begin{lemma}\label{lem:036}
		The assignment 
		\begin{equation}\label{eq:033}
			\varphi\otimes v^*\mapsto v^*\varphi
		\end{equation}
		defines a $ K $-equivariant linear isomorphism
		\[ \Psi_{\rho,\Gamma}:\mathcal A_{cusp}(\Gamma\backslash\Sp_{2n}(\R),\rho,V)^{\mathfrak p_\C^-}\otimes_\C V^*\to\mathcal A_{cusp}(\Gamma\backslash\Sp_{2n}(\R))^{\mathfrak p_\C^-}_{[\rho_K]}, \]
		where the action of $ K $ on $ \mathcal A_{cusp}(\Gamma\backslash\Sp_{2n}(\R),\rho,V)^{\mathfrak p_\C^-}\otimes_\C V^* $ is by $ 1\otimes \sigma_K^* $.
	\end{lemma}

	\begin{proof}
		The assignment \eqref{eq:033} is $ K $-equivariant by Lem.\,\ref{lem:029}\ref{lem:029:2}. Moreover, since $ \sigma_K^*\cong\rho_K $ by Lem.\ \ref{lem:012}\ref{lem:012:2}, it follows easily that the linear operator $ \Psi_{\rho,\Gamma} $ is well-defined, the only slightly non-trivial point being the $ \mathcal Z(\mathfrak g) $-finiteness of the function $ v^*\varphi $ given $ \varphi\in\mathcal A_{cusp}(\Gamma\backslash\Sp_{2n}(\R),\rho,V)^{\mathfrak p_\C^-} $ and $ v\in V $. To prove the latter, by the $ K $-equivariance of \eqref{eq:033}, we may assume that $ v $ is a highest weight vector for $ \rho $. Then, $ v^* $ is a highest weight vector for $ \sigma_K^* $ by Lem.\,\ref{lem:012}\ref{lem:012:3}, so $ \sigma_K^*(\mathfrak k_\C^+)v^*=0 $ by \S\ref{subsec:008}, hence, again by the $ K $-equivariance of \eqref{eq:033}, $ v^*\varphi $ is annihilated by $ \mathfrak k_\C^+ $. As $ v^*\varphi $ is moreover $ K $-finite on the right by Lem.\,\ref{lem:029}\ref{lem:029:2} and annihilated by $ \mathfrak p_\C^- $ by \ref{enum:022:2}, it is $ \mathcal Z(\mathfrak g) $-finite since $ \mathcal Z(\mathfrak g)\subseteq \mathcal U(\mathfrak k)\oplus\mathcal U(\mathfrak g)(\mathfrak p_\C^-\oplus\mathfrak k_\C^+) $ by \cite[Lem.\,8.17]{knapp86}.
		
		Next, to prove that $ \Psi_{\rho,\Gamma} $ is bijective, we note that $ \mathcal A_{cusp}(\Gamma\backslash\Sp_{2n}(\R))^{\mathfrak p_\C^-}_{[\rho_K]} $ has a direct sum decomposition
		\begin{equation}\label{eq:035}
			\mathcal A_{cusp}(\Gamma\backslash\Sp_{2n}(\R))^{\mathfrak p_\C^-}_{[\rho_K]}=\bigoplus_{r}U_r
		\end{equation}
		into irreducible $ K $-invariant subspaces $ U_r\overset K\cong (\rho_K,V)\overset K\cong(\sigma_K^*,V^*) $ (see, e.g., \cite[Prop.\,1.18(a)]{knapp_vogan95}). For each $ r $, we fix a $ K $-equivalence $ E_r:(\sigma_K^*,V^*)\to U_r $ and define a function $ \varphi_r:\Sp_{2n}(\R)\to V $ by requiring that for every $  g\in\Sp_{2n}(\R) $,
		\begin{equation}\label{eq:034}
			\scal{\varphi_r(g)}v=\left(E_r(v^*)\right)(g),\qquad v\in V.
		\end{equation}
		The function $ \varphi_r $ has the property \ref{enum:022:3} since for all $ k\in K $ and $ v\in V $, we have
 		\[ \begin{aligned}
 			\scal{\varphi_r(\spacedcdot k)}v
 			&\overset{\eqref{eq:034}}=R(k)E_r(v^*)
 			=E_r\sigma_K^*(k)v^*
 			=E_r\left(\sigma_K(k)v\right)^*\\
 			&\overset{\eqref{eq:034}}=\scal{\varphi_r(\spacedcdot)}{\sigma_K(k)v}
 			=\scal{\sigma_K(k)^{-1}\varphi_r(\spacedcdot)}{v},
 		\end{aligned} \]
 		where the last equality holds by the unitarity of $ \sigma_K $. Since moreover, for any $ v\in V $, the function $ \scal{\varphi_r(\spacedcdot)}v $ belongs to $ \mathcal A_{cusp}(\Gamma\backslash\Sp_{2n}(\R))^{\mathfrak p_\C^-} $, it follows that the function $ \varphi_r $ belongs to $ \mathcal A_{cusp}(\Gamma\backslash\Sp_{2n}(\R),\rho,V)^{\mathfrak p_\C^-} $. Thus, by \eqref{eq:034}, the operator $ \Psi_{\rho,\Gamma} $ restricts to a $ K $-equivalence 
 		\begin{equation}\label{eq:036}
 			\left(1\otimes\sigma_K^*,\C\varphi_r\otimes_\C V^*\right)\xrightarrow{\sim}\left(R,U_r\right).
 		\end{equation}
 		In particular, by \eqref{eq:035}, $ \Psi_{\rho,\Gamma} $ is surjective, the functions $ \varphi_r $ are linearly independent, and to prove that $ \Psi_{\rho,\Gamma} $ is injective, it suffices to prove that the functions $ \varphi_r $ span $ \mathcal A_{cusp}(\Gamma\backslash\Sp_{2n}(\R),\rho,V)^{\mathfrak p_\C^-} $. To this end, let $ \varphi\in\mathcal A_{cusp}(\Gamma\backslash\Sp_{2n}(\R),\rho,V)^{\mathfrak p_\C^-} $, and fix a highest weight vector $ v^{top} $ for $ \rho $. As in the first part of the proof, the function $ \left(v^{top}\right)^*\varphi\in\mathcal A_{cusp}(\Gamma\backslash\Sp_{2n}(\R))^{\mathfrak p_\C^-}_{[\rho_K]} $ is annihilated by $ \mathfrak k_\C^+ $, hence by \eqref{eq:035} it belongs to
 		\[	\bigoplus_r U_r^{\mathfrak k_\C^+}
 			\overset{\eqref{eq:036}}=\bigoplus_r\C\,\Psi_{\rho,\Gamma}\left(\varphi_r\otimes \left(v^{top}\right)^*\right)
 			=\bigoplus_r\C\left(v^{top}\right)^*\varphi_r. \]
 		By Lem.\,\ref{lem:029}\ref{lem:029:1}, it follows that $ \varphi $ belongs to $ \sum_r\C\varphi_r $, which proves the claim.
 	\end{proof}

	Let us fix a highest weight vector $ v^{top} $ for $ \rho $. By restricting the operator $ \Psi_{\rho,\Gamma} $ of Lem.\,\ref{lem:036} to the subspace of $ \mathfrak k_\C^+ $-invariants,	\[ \mathcal A_{cusp}(\Gamma\backslash\Sp_{2n}(\R),\rho,V)^{\mathfrak p_\C^-}\otimes_\C\, \C\left(v^{top}\right)^*, \]
	we obtain the following corollary.

	\begin{corollary}\label{cor:051}
		The assignment $ \varphi\mapsto\left(v^{top}\right)^*\varphi $
		defines a linear isomorphism 
		 \begin{equation}\label{eq:114}
		 	\mathcal A_{cusp}(\Gamma\backslash\Sp_{2n}(\R),\rho,V)^{\mathfrak p_\C^-}\cong
		 	\left(\mathcal A_{cusp}(\Gamma\backslash\Sp_{2n}(\R))_{[\rho_K]}^{\mathfrak p_\C^-}\right)^{\mathfrak k_\C^+},
		 \end{equation}
	 	where, on the right-hand side, we employ the notation \eqref{eq:113} for the subspace of $ \mathfrak k_\C^+ $-invariants.
	\end{corollary}

	The space $ \mathcal A_{cusp}(\Gamma\backslash\Sp_{2n}(\R))^{\mathfrak p_\C^-}_{[\rho_K]} $ has the following description in terms of the spectral decomposition of $ L^2_{cusp}(\Gamma\backslash\Sp_{2n}(\R)) $.
	
	\begin{lemma}\label{lem:111}
		We have 
		\begin{equation}\label{eq:038}
			\mathcal A_{cusp}(\Gamma\backslash\Sp_{2n}(\R))^{\mathfrak p_\C^-}_{[\rho_K]}=\left(L^2_{cusp}(\Gamma\backslash\Sp_{2n}(\R))_{[\pi_\rho^*]}\right)_{[\rho_K]}, 
		\end{equation}
		where, on the right-hand side, we employ twice the notation $ H_{[\tau]} $, introduced in \S\ref{subsec:110}, for the $ \tau $-isotypic component of a unitary representation $ H $.
	\end{lemma}

	\begin{proof}[Proof of Lem.\,\ref{lem:111}]
		Let	$ \varphi\in\mathcal A_{cusp}(\Gamma\backslash\Sp_{2n}(\R))^{\mathfrak p_\C^-}_{[\rho_K]}\setminus\left\{0\right\} $, and let $ H $ denote the smallest closed $ \Sp_{2n}(\R) $-invariant subspace of $ L^2_{cusp}(\Gamma\backslash\Sp_{2n}(\R)) $ containing $ \varphi $. By Harish-Chandra's result \cite[Lem.\,77]{harish66} (see also \cite[Lem.\,4.2]{zunar23}), there exist $ m\in\Z_{>0} $ and mutually orthogonal irreducible closed $ \Sp_{2n}(\R) $-invariant subspaces $ H_1,\ldots,H_m $ of $ H $ such that $ H=\bigoplus_{r=1}^mH_r $. For each $ r\in\left\{1,\ldots,m\right\} $, the orthogonal projection $ p_r:H\twoheadrightarrow H_r $ is $ \Sp_{2n}(\R) $-equivariant, hence $ \varphi_r:=p_r(\varphi) $ is annihilated by $ \mathfrak p_\C^- $ and generates a non-trivial $ K $-invariant subspace $ \left<\varphi_r\right>_K $ of $ (H_r)_{[\rho_K]} $ that is also annihilated by $ \mathfrak p_\C^- $ since $ \mathrm{Ad}(K)\mathfrak p_\C^-=\mathfrak p_\C^- $. By Lem.\,\ref{lem:011}\ref{lem:011:2}, it follows that $ H_r $ is $ \Sp_{2n}(\R) $-equivalent to $ \pi_\rho^* $. Thus, 
		\[ \varphi=\sum_r\varphi_r\in\sum_r\left(H_r\right)_{[\rho_K]}\subseteq\left(L^2_{cusp}(\Gamma\backslash\Sp_{2n}(\R))_{[\pi_\rho^*]}\right)_{[\rho_K]}. \]
		This proves the left-to-right inclusion in \eqref{eq:038}. The reverse inclusion holds since by Lem.\,\ref{lem:037} and Lem.\,\ref{lem:011}\ref{lem:011:2}, the $ \rho_K $-isotypic component of $ \pi_\rho^* $ is annihilated by $ \mathfrak p_\C^- $ (see also \cite[end of \S2.2]{borel_jacquet79}).
	\end{proof}
	
	\section{K-finite matrix coefficients and their Poincaré series}\label{sec:105}
	
	Given a finite-dimensional complex Hilbert space $ U $ and a function $ \varphi\in L^1(\Sp_{2n}(\R), U) $, one proves as in \cite[beginning of \S4]{muic09} that the Poincaré series
	\begin{equation}\label{eq:118}
		P_\Gamma\varphi=\sum_{\gamma\in\Gamma}\varphi(\gamma\spacedcdot)
	\end{equation}
	converges absolutely almost everywhere on $ \Sp_{2n}(\R) $ and defines an element of $ L^1(\Gamma\backslash\Sp_{2n}(\R),U) $. Moreover, as a special case of Mili\v ci\'c's result \cite[Lem.\,6.6]{muic19}, by \cite[proof of Thm.\,3.10(i)]{muic09} and \cite[Lem.\,5.1]{zunar23} we have the following lemma.
	
	\begin{lemma}\label{lem:040}
		Let $ (\pi,H) $ be an integrable representation of $ \Sp_{2n}(\R) $. Given $  h,h'\in H_K $, let $ c_{h,h'} $ be the $ K $-finite matrix coefficient of $ \pi $ given by \eqref{eq:039}. Then, we have the following:
		\begin{enumerate}[label=\textup{(\roman*)}]
			\item\label{lem:040:1} The Poincaré series $ P_\Gamma c_{h,h'} $ converges absolutely and uniformly on compact subsets of $ \Sp_{2n}(\R) $.
			\item\label{lem:040:2} We have
			\[  \left(L^2_{cusp}(\Gamma\backslash\Sp_{2n}(\R))_{[\pi]}\right)_K=\mathrm{span}_\C\left\{P_\Gamma c_{h,h'}:h,h'\in H_K\right\}.  \]
		\end{enumerate}
	\end{lemma}
	
	In particular, if $ \omega_n>2n $, we have the following description of the space \eqref{eq:038}.
	
	\begin{proposition}\label{prop:115}
		Suppose that $ \omega_n>2n $. Then, we have
		 \begin{equation}\label{eq:043}
		 	\left(L^2_{cusp}(\Gamma\backslash\Sp_{2n}(\R))_{[\pi_\rho^*]}\right)_{[\rho_K]}=\mathrm{span}_\C\left\{P_\Gamma c_{f_{1,v}^*,f^*}:v\in V\text{ and }f\in\calH^2(\rho)_K\right\}.
		 \end{equation}
	\end{proposition}
	
	\begin{proof}
		By Lem.\,\ref{lem:041}, the representation $ \pi_\rho^* $ is integrable, hence Lem.\,\ref{lem:040}\ref{lem:040:2} and the argument of \cite[proof of Prop.\,5.2]{zunar23} imply that
		\[ \left(L^2_{cusp}(\Gamma\backslash\Sp_{2n}(\R))_{[\pi_\rho^*]}\right)_{[\rho_K]}=\mathrm{span}_\C\left\{P_\Gamma c_{h,h'}:h\in\left(\calH^2(\rho)^*\right)_{[\rho_K]},h'\in\left(\calH^2(\rho)^*\right)_K\right\}, \]
		from which \eqref{eq:043} follows by \eqref{eq:112}.
	\end{proof}

	\begin{corollary}\label{cor:106}
		Suppose that $ \omega_n>2n $. Let $ v^{top} $ be a highest weight vector for $ \rho $. Then, 
		\begin{equation}\label{eq:052}
			\left(\left(L^2_{cusp}(\Gamma\backslash\Sp_{2n}(\R))_{[\pi_\rho^*]}\right)_{[\rho_K]}\right)^{\mathfrak k_\C^+}=\left\{P_\Gamma c_{f_{1,v^{top}}^*,f^*}:f\in\calH^2(\rho)_K\right\},
		\end{equation}
		where $ f_{1,v^{top}}\in\calH^2(\rho)_K $ is defined by \eqref{eq:048}.
	\end{corollary}
	
	\begin{proof}
		It follows easily from Lem.\,\ref{lem:116}\ref{lem:116:2} that for every $ f\in\calH^2(\rho)_K $, the linear operator 
		\[ \left(\sigma_K^*,V^*\right)\to L^2_{cusp}(\Gamma\backslash\Sp_{2n}(\R)),\qquad
		 v^*\mapsto P_\Gamma c_{f_{1,v}^*,f^*}, \]
		is $ K $-equivariant. Thus, \eqref{eq:043} implies that for every weight $ \lambda $ of $ \rho_K $, we have
		\begin{equation}\label{eq:115}
			\left(\left(L^2_{cusp}(\Gamma\backslash\Sp_{2n}(\R))_{[\pi_\rho^*]}\right)_{[\rho_K]}\right)_\lambda=\mathrm{span}_\C\left\{P_\Gamma c_{f_{1,v}^*,f^*}:v^*\in \left(V^*\right)_\lambda,f\in\calH^2(\rho)_K\right\},
		\end{equation}
		where we use the notation $ H_\lambda $ for the weight $ \lambda $ subspace of a locally $ K $-finite representation $ H $ of $ K $ (cf.\ \cite[Prop.\,1.18(a)]{knapp_vogan95} and \cite[(4.4)]{knapp86}). In the case when $ \lambda $ is the highest weight $ \sum_{r=1}^n\omega_re_r $ of $ \rho_K\cong\sigma_K^* $ (see Lem.\,\ref{lem:012}\ref{lem:012:1}--\ref{lem:012:2}), this equality is \eqref{eq:052} by \S\ref{subsec:008} and Lem.\,\ref{lem:012}\ref{lem:012:3}.
	\end{proof}
	
	Suppose that $ \omega_n>n $. We define a positive constant
	\[ C_\rho=\norm{f_{1,v^{top}}}_{\calH^2(\rho)}^2=\norm{p_{1,v^{top}}}_{\calD^2(\rho)}^2
	\overset{\eqref{eq:117}}=\int_{\calD_n}\norm{\rho\left(I_n-w^*w\right)^{\frac12}v^{top}}^2\,d\mathsf v_\calD(w), \]
	where $ v^{top} $ is any highest weight vector for $ \rho $ such that $ \norm{v^{top}}=1 $.
	In the following proposition, we use the notation \eqref{eq:032} and, given a function $ \varphi $ with domain $ \Sp_{2n}(\R) $, write $ \check\varphi=\varphi\left(\spacedcdot^{-1}\right) $.
	
	\begin{proposition}\label{prop:049}
		Suppose that $ \omega_n>n $.
		Let $ f\in\calH^2(\rho)_K $ and $ v\in V $. Then, we have the following:
		\begin{enumerate}[label=\textup{(\roman*)}]
			\item\label{prop:049:1} The $ K $-finite matrix coefficient $ c_{f,f_{1,v}} $ of $ \pi_\rho $ is given by
			\[ c_{f,f_{1,v}}=C_\rho\, v^*\check F_f. \]
			\item\label{prop:049:2} The $ K $-finite matrix coefficient $ c_{f_{1,v}^*,f^*} $ of $ \pi_\rho^* $ is given by
			 \begin{equation}\label{eq:054}
			 	c_{f_{1,v}^*,f^*}=C_\rho\, v^*F_f. 
			 \end{equation}
		\end{enumerate}
	\end{proposition}
	
	\begin{proof}
		We follow the idea of the proofs of \cite[Lem.\,3-5]{muic10} and \cite[Prop.\,5.3]{zunar23}. First, we note that the assignment $ f\mapsto F_f $ defines a unitary $ \Sp_{2n}(\R) $-equivalence $ E_\rho $ from $ \left(\pi_\rho,\calH^2(\rho)\right) $ onto an irreducible closed $ \Sp_{2n}(\R) $-subrepresentation $ H_\rho $ of the left regular representation $ L $ of $ \Sp_{2n}(\R) $ on $ L^2(\Sp_{2n}(\R),V) $.
		
		Next, let $ f\in\calH^2(\rho)_K\setminus\left\{0\right\} $. The function $ F_f\in\left(H_\rho\right)_K $ is $ \mathcal Z(\mathfrak g) $-finite and, by Lem.\,\ref{lem:021}\ref{lem:021:3}, right $ K $-finite, hence by \cite[Thm.\,1]{harish66} there exists a smooth, compactly supported function $ \alpha:\Sp_{2n}(\R)\to\C $ such that $ F_f=F_f*\check\alpha $. Thus, denoting by $ \mathrm{pr}_{H_\rho} $ the orthogonal projection from $ L^2(\Sp_{2n}(\R),V) $ onto $ H_\rho $, we have
		\[ \begin{aligned}
			\left(v^*\check F_f\right)(g)
			&=\scal{\left(F_f*\check\alpha\right)\left(g^{-1}\right)}v\\
			&=\int_{\Sp_{2n}(\R)}\alpha(g')\scal{F_f\left(g^{-1}g'\right)}v\,dg'\\
			&=\scal{L(g)F_f}{\overline\alpha v}_{L^2(\Sp_{2n}(\R),V)}\\
			&=\scal{L(g)F_f}{\mathrm{pr}_{H_\rho}\left(\overline\alpha v\right)}_{H_\rho},\qquad v\in V,\ g\in\Sp_{2n}(\R).
		\end{aligned} \]
		Applying the inverse of the unitary $ \Sp_{2n}(\R) $-equivalence $ E_\rho $, we conclude that there exists a function $ h:V\to\calH^2(\rho) $ such that
		\begin{equation}\label{eq:045}
			 v^*\check F_f=c_{f,h(v)},\qquad v\in V. 
		\end{equation}
	
		One shows easily that $ h $ is a $ K $-equivariant linear operator from $ (\sigma_K,V) $ to $ \calH^2(\rho)_{[\sigma_K]} $. For example, to prove the $ K $-equivariance, we note that for all $ k\in K $ and $ v\in V $, we have
		\[ \begin{aligned}
			&\scal{\pi_\rho(g)f}{h(\sigma_K(k)v)}_{\calH^2(\rho)}
			\overset{\eqref{eq:045}}=\scal{\check F_f(g)}{\sigma_K(k)v}
			=\scal{\sigma_K(k)^{-1}F_f\big(g^{-1}\big)}{v}\\
			&\qquad=\scal{F_f\big(g^{-1}k\big)}{v}
			=\scal{\check F_f\big(k^{-1}g\big)}{v}
			=\left(L(k)v^*\check F_f\right)(g)
			\overset{\eqref{eq:045}}=\left(L(k)c_{f,h(v)}\right)(g)\\
			&\qquad=\scal{\pi_\rho\big(k^{-1}g\big)f}{h(v)}_{\calH^2(\rho)}
			=\scal{\pi_\rho(g)f}{\pi_\rho(k)h(v)}_{\calH^2(\rho)},\quad g\in\Sp_{2n}(\R),
		\end{aligned} \]
		where the third equality holds by Lem.\,\ref{lem:021}\ref{lem:021:3}. Since by the irreducibility of $ \pi_\rho $, the linear span of $ \pi_\rho(\Sp_{2n}(\R))f $ is dense in $ \calH^2(\rho) $, it follows that 
		\[h(\sigma_K(k)v)=\pi_\rho(k)h(v),\qquad k\in K,\ v\in V, \]
		i.e., $ h $ is $ K $-equivariant. Thus, by Lem.\,\ref{lem:116}\ref{lem:116:1} and Schur's lemma \cite[Prop.\,1.5]{knapp86}, it follows that $ h $ is, up to a scalar multiple, given by the assignment $ v\mapsto f_{1,v} $. 
		Thus, \eqref{eq:045} implies that there exists a function $ \lambda:\calH^2(\rho)_K\to\C $ such that
		\begin{equation}\label{eq:046}
			c_{f,f_{1,v}}=\lambda(f)\,v^*\check F_f,\qquad f\in\calH^2(\rho)_K,\ v\in V. 
		\end{equation}
		Let us fix a unit vector $ v^{top} $ of highest weight for $ \rho $. By Lem.\,\ref{lem:029}\ref{lem:029:1}, the assignment
		\[ f\mapsto c_{f,f_{1,v^{top}}}\overset{\eqref{eq:046}}=\left(v^{top}\right)^*\check F_{\lambda(f)f}\mapsto\lambda(f)f \]
		defines a linear operator on $ \calH^2(\rho)_K $, hence $ \lambda $ is constant on $ \calH^2(\rho)_K\setminus\left\{0\right\} $; more precisely,  for all $ f\in\calH^2(\rho)_K\setminus\left\{0\right\} $, we have
		\begin{equation}\label{eq:050}
			 \lambda(f)=\left(\frac{c_{f_{1,v^{top}},{f_{1,v^{top}}}}}{\left(v^{top}\right)^*\check F_{f_{1,v^{top}}}}\right)(1_{\Sp_{2n}(\R)})
		=\frac{\norm{f_{1,v^{top}}}_{\calH^2(\rho)}^2}{\norm{v^{top}}^2}=\frac{C_\rho}{1^2}=C_\rho, 
		\end{equation}
		where the first equality is obtained by letting $ f=f_{1,v^{top}} $ and $ v=v^{top} $ in \eqref{eq:046} and evaluating at $ 1_{\Sp_{2n}(\R)} $, and the second equality holds by \eqref{eq:039}, \eqref{eq:047}, and \eqref{eq:048}. The equalities \eqref{eq:046} and \eqref{eq:050} imply the claim \ref{prop:049:1}, which immediately implies \ref{prop:049:2}.
	\end{proof}
	
	By Cartan's KAK decomposition \cite[Thm.\,5.20]{knapp86}, given $ g\in\Sp_{2n}(\R) $, there exist matrices $ k,k'\in K $ and $ t\in\R^n $ such that
	\begin{equation}\label{eq:065}
		g=kh_tk',
	\end{equation}
	where $ h_t=\mathrm{diag}\left(e^{t_1},\ldots,e^{t_n},e^{-t_1},\ldots,e^{-t_n}\right)\in\mathrm{Diag}_{2n}(\R_{>0}) $. By \cite[Prop.\,5.28]{knapp86}, there exists $ M_0\in\R_{>0} $ such that for every $ f\in C_c(\Sp_{2n}(\R)) $, the Haar integral $ \int_{\mathrm{Sp_{2n}(\R)}}f(g)\,dg $ equals
	\begin{equation}\label{eq:067}
		M_0\int_{K}\int_{A^+}\int_Kf(kh_tk')\left(\prod_{1\leq r<s\leq n}\sinh(t_r-t_s)\right)\left(\prod_{1\leq r\leq s\leq n}\sinh(t_r+t_s)\right)\,dk\,dt\,dk',
	\end{equation}
	where $ A^+=\left\{t=(t_1,\ldots,t_n)\in\R^n:t_1>\ldots>t_n>0\right\} $.
	Moreover, a simple expression for matrix coefficients of Prop.\,\ref{prop:049} is implied by the following lemma.
	
	\begin{lemma}
		Suppose that $ \omega_n>n $.
		Let $ \mu\in\C[X_{r,s}:1\leq r,s\leq n] $ and $ v\in V $. Denoting $ d_t=\mathrm{diag}(t_1,\ldots,t_n) $ for $ t\in\R^n $, we have
		\begin{equation}\label{eq:066}
			F_{f_{\mu,v}}(k_uh_tk_{u'})=\mu\left(u\tanh(d_t)u^\top\right)\rho\left(\left(u'\right)^{\top}\cosh(d_t)^{-1}u^{\top}\right)v 
		\end{equation}
		for all $ u,u'\in\mathrm{U}(n) $ and $ t\in\R^n $.
	\end{lemma}
	
	\begin{proof}
		A proof in the scalar-valued case (i.e., the case when $ \dim_\C V=1 $) is given in \cite[proof of Lem.\,6.2]{zunar23} and generalizes in an elementary way to the vector-valued case.
	\end{proof}

	For every function $ f:\calH_n\to V $, we define a Poincar\'e series
	\[ P_{\Gamma,\rho}f=\sum_{\gamma\in\Gamma}f\big|_\rho\gamma. \]
	The following vector-valued version of \cite[Lem.\,5.4]{zunar23} is elementary.
 	
 	\begin{lemma}\label{lem:055}
 		Let $ f:\calH_n\to V $. Then, the series $ P_{\Gamma,\rho}f $ converges absolutely and uniformly on compact subsets of $ \calH_n $ if and only if the series $ P_\Gamma F_f $ converges absolutely and uniformly on compact subsets of $ \Sp_{2n}(\R) $. Furthermore, assuming convergence, the following holds:
 		\begin{enumerate}[label=\textup{(\roman*)}]
 			\item We have
 			\begin{equation}\label{eq:056}
 				 F_{P_{\Gamma,\rho}f}=P_\Gamma F_f. 
 			\end{equation}
 			\item The function $ P_{\Gamma,\rho}f $ vanishes identically if and only if $ P_\Gamma F_f $ does.
 		\end{enumerate}
 	\end{lemma}
		
	By \eqref{eq:061}, the following theorem is a restatement of Thm.\,\ref{thm:100}.
		
	\begin{theorem}\label{thm:062}
		Suppose that $ \omega_n>2n $. Then, for every $ f\in\calH^2(\rho)_K $, the Poincar\'e series $ P_{\Gamma,\rho}f $ converges absolutely and uniformly on compact subsets of $ \calH_n $. Moreover, we have
		\begin{equation}\label{eq:059}
			S_\rho(\Gamma)=\left\{P_{\Gamma,\rho}f:f\in\calH^2(\rho)_K\right\}.
		\end{equation}
	\end{theorem}

	\begin{proof}
		Let $ f\in\calH^2(\rho)_K $. It follows from Lem.\,\ref{lem:041}, Lem.\,\ref{lem:040}\ref{lem:040:1}, and Prop.\,\ref{prop:049}\ref{prop:049:2} that the Poincar\'e series $ P_\Gamma\,v^*F_f $ converges absolutely and uniformly on compact subsets of $ \Sp_{2n}(\R) $ for every $ v\in V $, hence so does the Poincar\'e series $ P_\Gamma F_f $. Consequently, by Lem.\,\ref{lem:055}, the Poincar\'e series $ P_{\Gamma,\rho}f $ converges absolutely and uniformly on compact subsets of $ \calH_n $.
		
		Let $ v^{top} $ be a highest weight vector for $ \rho $. By Cor.\,\ref{cor:051}, we have
		 \begin{equation}\label{eq:057}
		 	\begin{aligned}
			\left(v^{top}\right)^*\left(\mathcal A_{cusp}(\Gamma\backslash\Sp_{2n}(\R),\rho,V)^{\mathfrak p_\C^-}\right)
			&\overset{\phantom{\eqref{eq:054}}}=\left(\mathcal A_{cusp}(\Gamma\backslash\Sp_{2n}(\R))_{[\rho_K]}^{\mathfrak p_\C^-}\right)^{\mathfrak k_\C^+}\\
			&\overset{\eqref{eq:038}}=\left(\left(L^2_{cusp}(\Gamma\backslash\Sp_{2n}(\R))_{[\pi_\rho^*]}\right)_{[\rho_K]}\right)^{\mathfrak k_\C^+}\\
			&\overset{\eqref{eq:052}}=\left\{P_\Gamma c_{f_{1,v^{top}}^*,f^*}:f\in\calH^2(\rho)_K\right\}\\
			&\overset{\eqref{eq:054}}=\left\{P_\Gamma \left(v^{top}\right)^*F_f:f\in\calH^2(\rho)_K\right\}\\
			&\overset{\phantom{\eqref{eq:054}}}=\left(v^{top}\right)^*\left\{P_\Gamma F_f:f\in\calH^2(\rho)_K\right\}.
		\end{aligned} 
		 \end{equation}
		Since the elements of both $ \mathcal A_{cusp}(\Gamma\backslash\Sp_{2n}(\R),\rho,V)^{\mathfrak p_\C^-} $  and $ \left\{P_\Gamma F_f:f\in\calH^2(\rho)_K\right\} $ satisfy \ref{enum:022:3} (see Lem.\,\ref{lem:021}\ref{lem:021:3}), by Lem.\,\ref{lem:029}\ref{lem:029:1} the equality \eqref{eq:057} implies that
		\[ \mathcal A_{cusp}(\Gamma\backslash\Sp_{2n}(\R),\rho,V)^{\mathfrak p_\C^-}=\left\{P_\Gamma F_f:f\in\calH^2(\rho)_K\right\}. \]
		By applying the inverse of the unitary isomorphism $ \Phi_{\rho,\Gamma} $ of Lem.\,\ref{lem:058}, by \eqref{eq:056} we obtain \eqref{eq:059}.
	\end{proof}

	\section{Relation to the reproducing kernel function for $ S_\rho(\Gamma) $}\label{sec:007}
	
	Suppose that $ \omega_n>2n $. The reproducing kernel function for $ S_\rho(\Gamma) $ is the unique function 
	\[ K_\rho^\Gamma:\calH_n\times\calH_n\to\mathrm{End}_\C(V) \] 
	such that for all $ z\in\calH_n $ and $ v\in V $, the following holds:
	\begin{enumerate}
		\item The function $ K_\rho^\Gamma(\spacedcdot,z)v:\calH_n\to V $ belongs to $ S_\rho(\Gamma) $.
		\item $ \scal{f(z)}v=\scal{f}{K_\rho^\Gamma(\spacedcdot,z)v}_{S_\rho(\Gamma)} $ for all $ f\in S_\rho(\Gamma) $.
	\end{enumerate}
	By \cite[p.\ 10-30]{godement57_10}, there exists a constant $ c_\rho\in\C^\times $, depending only on $ \rho $, such that
	\[ K_\rho^\Gamma(\spacedcdot,z)v=c_\rho P_{\Gamma,\rho}K_{\rho,z,v},\qquad z\in\calH_n,\ v\in V, \]
	where $ K_{\rho,z,v} $ is the function $ \calH_n\to V $ defined by
	\[ K_{\rho,z,v}(\zeta)=\rho\left(\frac1{2i}\left(\zeta-\overline z\right)\right)^{-1}v,\qquad\zeta\in\calH_n. \] 
	Since by \eqref{eq:048}, $ K_{\rho,iI_n,v}=f_{1,v} $ for every $ v\in V $, it follows that
	\[ K_\rho^\Gamma(\spacedcdot,iI_n)v=c_\rho P_{\Gamma,\rho}f_{1,v},\qquad v\in V. \]
	In other words, Lem.\,\ref{lem:107} holds.

	\section{A non-vanishing result}\label{sec:008}
	
	In this section, we prove a non-vanishing result for Siegel cusp forms $ P_{\Gamma,\rho}f_{\mu,v} $ of Thm.\,\ref{thm:100} using the following variant of Mui\'c's integral non-vanishing criterion \cite[Thm.\,4.1]{muic09}.
	
	\begin{proposition}\label{prop:063}
		Let $ \Gamma $ be a discrete subgroup of $ \Sp_{2n}(\R) $. Let $ \varphi\in L^1(\Sp_{2n}(\R),U) $, where $ U $ is a finite-dimensional complex Hilbert space. Assume that there exists a Borel measurable subset $ C\subseteq\Sp_{2n}(\R) $ with the following properties:
			\begin{enumerate}[label=\textup{(C\arabic*)}]
				\item\label{enum:063:1} $ CC^{-1}\cap\Gamma=\left\{1_{\mathrm{Sp}_{2n}(\R)}\right\} $.
				\item\label{enum:063:2} $ \int_C\norm{\varphi(g)}_U\,dg>\frac12\int_{\Sp_{2n}(\R)}\norm{\varphi(g)}_U\,dg $.
			\end{enumerate}
		Then, the Poincar\'e series $ P_\Gamma\varphi $ defined by \eqref{eq:118} converges absolutely almost everywhere on $ \Sp_{2n}(\R) $, and $ P_\Gamma\varphi\in L^1(\Gamma\backslash\Sp_{2n}(\R),U)\setminus\left\{0\right\} $.
	\end{proposition}
	
	\begin{proof}
		A proof of convergence (resp., non-vanishing) is obtained by replacing every occurence of the complex absolute value symbol $ \abs\spacedcdot $ in \cite[beginning of \S4]{muic09} (resp., \cite[proof of Thm.\,5.3]{zunar19}) by $ \norm\spacedcdot_U $.
	\end{proof}
	
	Following \cite{zunar23}, we define a function $ M:\Z_{>0}\to\left[0,1\right] $,
	\[ M(N)=\left(\sqrt{1+\frac{4n}{N^2}}+\sqrt{\frac{4n}{N^2}}\right)^{-2}, \]
	and, for every $ t\in\R_{>0} $, let
	\[ A_t^+=\left\{x=(x_1,\ldots,x_n)\in\mathbb R^n:t>x_1>\ldots>x_n>0\right\}. \]
	We have the following generalization of \cite[Prop.\,6.5]{zunar23} from the scalar-valued to the vector-valued case.

	\begin{theorem}\label{thm:101}
		Suppose that $ \omega_n>2n $. Let $ \mu\in\C[X_{r,s}:1\leq r,s\leq n]\setminus\left\{0\right\} $ and $ v\in V\setminus\left\{0\right\} $. Let $ \varphi_{\rho,\mu,v}:\mathrm U(n)\times A_1^+\to\R_{\geq0} $,
		\[ \varphi_{\rho,\mu,v}(u,x)=\abs{\mu\left(ud_x^{\frac12}u^\top\right)}\,\norm{\rho\left((I_n-d_x)^{\frac12}u^{\top}\right)v}\,\frac{\prod_{1\leq r<s\leq n }(x_r-x_s)}{\prod_{r=1}^n(1-x_r)^{n+1}}, \]
		where $ d_x=\mathrm{diag}(x_1,\ldots,x_n) $.
		Let $ N_0=N_0(\rho,\mu,v) $ be the smallest integer $ N\in\Z_{\geq3} $ such that
		\begin{equation}\label{eq:064}
			\int_{A_{M(N)}^+}\int_{\mathrm U(n)}\varphi_{\rho,\mu,v}(u,x)\,du\,dx>\frac12\int_{A_1^+}\int_{\mathrm U(n)}\varphi_{\rho,\mu,v}(u,x)\,du\,dx, 
		\end{equation}
		where $ du $ is the Haar measure on $ \mathrm U(n) $ such that $ \int_{\mathrm U(n)}du=1 $. Then, for every $ N\in\Z_{\geq N_0} $ and for every subgroup $ \Gamma $ of finite index in $ \Gamma_n(N) $, we have that
		\[ P_\Gamma F_{f_{\mu,v}}\not\equiv0\qquad\text{and}\qquad P_{\Gamma,\rho}f_{\mu,v}\not\equiv0. \]
	\end{theorem}
	
	\begin{proof}
		We will apply Prop.\,\ref{prop:063} to the Poincar\'e series 
		\[ P_\Gamma F_{f_{\mu,v}}=F_{P_{\Gamma,\rho}f_{\mu,v}}, \]
		with the set $ C $ of the form
		\[ C_R=K\left\{h_t:t\in\left[0,R\right[^n\right\}K \]
		(see \eqref{eq:065}) for a suitable $ R\in\R_{>0} $. 
		
		Let $ N\in\Z_{\geq3} $. Since $ \Gamma_n(N)\cap K=\left\{1_{\mathrm{Sp}_{2n}(\R)}\right\} $, the lemmas \cite[Lemmas 6.3 and 6.4]{zunar23} imply that, for every subgroup $ \Gamma $ of $ \Gamma_n(N) $, the set $ C_R $ satisfies the condition \ref{enum:063:1} of Prop.\,\ref{prop:063} if the inequality
		\begin{equation}\label{eq:068}
			\tanh^2R\leq M(N) 
		\end{equation}
		holds. In particular, the set $ C_{\artanh \sqrt{M(N)}} $ satisfies \ref{enum:063:1} for every subgroup $ \Gamma $ of finite index in $ \Gamma_n(N) $.
		
		On the other hand, from the formulas \eqref{eq:067} and \eqref{eq:066}, by a computation analogous to the one in \cite[proof of Prop.\,6.5]{zunar23}, it follows that for every $ R\in\R_{>0} $, we have
		\begin{equation}\label{eq:069}
			\int_{C_R}\norm{F_{f_{\mu,v}}(g)}\,dg=M_0\int_{A^+_{\tanh^2R}}\int_{\mathrm U(n)}\varphi_{\rho,\mu,v}(u,x)\,du\,dx, 
		\end{equation}
		where $ M_0\in\R_{>0} $ is defined by \eqref{eq:067}.
		In particular, since $ F_{f_{\mu,v}}\in L^1(\Sp_{2n}(\R),V) $ by Lem.\,\ref{lem:041} and Prop.\,\ref{prop:049}\ref{prop:049:2}, the integrals in the inequality \eqref{eq:064} are finite. Thus, since $ M(N)\nearrow1 $ as $ N\to\infty $, by the monotone convergence theorem the integer $ N_0 $ is well-defined, and for every $ N\in\Z_{\geq N_0} $, the inequality \eqref{eq:064} holds, i.e., by \eqref{eq:069}, the set $ C_{\artanh \sqrt{M(N)}} $ satisfies the condition \ref{enum:063:2} of Prop.\,\ref{prop:063}. This finishes the proof of the theorem.
	\end{proof}

	\begin{remark*}
		While computationally demanding, determining the constants $ N_0(\rho,\mu,v) $ of Thm.\,\ref{thm:101} is feasible in some special cases. In \cite[Table 1]{zunar23}, the values $ N_0(\rho,\mu,v) $ are listed for the following values of arguments: $ \rho=\det^m:\GL_n(\C)\to\C^\times $ with $ n\in\left\{1,2\right\} $ and $ m\in\left\{2n+1,2n+2,\ldots,2n+8\right\} $, $ \mu=\det^l $ with $ l\in\left\{0,1,\ldots,12\right\} $, and $ v\neq0 $. 
	\end{remark*}

\bibliographystyle{amsplain}

\begin{thebibliography}{999999}
	\bibitem[AndZhu95]{andrianov_zhuravlev95} Andrianov, A.~N., Zhuravl\"ev, V.~G.:	\textit{Modular forms and Hecke operators.}	Translated from the 1990 Russian original by Neal Koblitz. Transl.\ Math.\ Monogr.\ \textbf{145}, American Mathematical Society, Providence, RI (1995) 
	
	\bibitem[AsgSch01]{asgari_schmidt01} Asgari, M., Schmidt, R.:	Siegel modular forms and representations. \textit{Manuscripta Math.}\ \textbf{104}(2), 173--200 (2001)
	
	%\bibitem[Ato18]{atobe18} Atobe, H.: Applications of Arthur's multiplicity formula to Siegel modular forms, preprint. arXiv:1810.09089 (2018)
	
	\bibitem[Ato19]{atobe19} Atobe, H.: Liftings of vector valued Siegel modular forms. In: Automorphic forms, automorphic representations and related topics. \textit{RIMS Kokyuroku} \textbf{2136}, 25--29, Research Institute for Mathematical Sciences, Kyoto (2019)
	
	
	\bibitem[BorJac79]{borel_jacquet79} Borel, A., Jacquet, H.:	Automorphic forms and automorphic representations. \textit{Automorphic forms, representations and L-functions} (Proc.\ Sympos.\ Pure Math., Oregon State Univ., Corvallis, Ore. (1977)), Part 1, Proc.\ Sympos.\ Pure Math.\ \textbf{XXXIII},
	American Mathematical Society, Providence, RI (1979), 189--207
	
	\bibitem[CheLan19]{chenevier_lannes19} Chenevier, G., Lannes, J.: \textit{Automorphic forms and even unimodular lattices.	Kneser neighbors of Niemeier lattices.} Translated from French by Reinie Ern\'e. Ergeb.\ Math.\ Grenzgeb.\ (3) \textbf{69}, Springer, Cham, 2019.
	
	\bibitem[CleFabGee20]{clery_faber_van_der_geer20} Cl\'ery, F., Faber, C., van der Geer, G.:	Concomitants of ternary quartics and vector-valued Siegel and Teichm\"uller modular forms of genus three. \textit{Selecta Math.\ (N.S.) }\textbf{26}(4), Paper No.\ 55, 39 pp. (2020)
	
	\bibitem[CleGee15]{clery_van_der_geer15} Cl\'ery, F., van der Geer, G.:	Constructing vector-valued Siegel modular forms from scalar-valued Siegel modular forms. \textit{Pure Appl.\ Math.\ Q.\ }\textbf{11}(1), 21--47 (2015)
	
	\bibitem[Fre79]{freitag79} Freitag, E.: Ein Verschwindungssatz für automorphe Formen zur Siegelschen Modulgruppe. \textit{Math.\ Z.}\ \textbf{165}(1), 11--18  (1979)
	
	\bibitem[Fre83]{freitag83} Freitag, E.:	\textit{Siegelsche Modulfunktionen.} Grundlehren Math.\ Wiss.\ \textbf{254}, Springer-Verlag, Berlin (1983)
	
	\bibitem[Fre91]{freitag91} Freitag, E.:	\textit{Singular modular forms and theta relations.} Lecture Notes in Math.\ \textbf{1487}, Springer-Verlag, Berlin (1991)
	
	\bibitem[Ful97]{fulton97} Fulton, W.: \textit{Young tableaux. With applications to representation theory and geometry}. London Math.\ Soc.\ Stud.\ Texts \textbf{35},
	Cambridge University Press, Cambridge (1997)
	
	\bibitem[Gee08]{van_der_geer08} Van der Geer, G.: Siegel modular forms and their applications. \textit{The 1-2-3 of modular forms.} Universitext, Springer-Verlag, Berlin (2008), 181--245
	
	\bibitem[Gel73]{gelbart73} Gelbart, S.:	Holomorphic discrete series for the real symplectic group. \textit{Invent.\ Math.}\ \textbf{19}, 49--58 (1973)
	
	\bibitem[God57-5]{godement57_5} Godement, R.: \textit{O\`u l'on g\'en\'eralise une int\'egrale \'etudi\'ee par C.\ L.\ Siegel, et g\'en\'eralisant la fonction $\Gamma$}.	S\'eminaire Henri Cartan, tome 10, no 1, exp.\ no 5, p.\ 1--24  (1957--1958)
	
	\bibitem[God57-7]{godement57_7} Godement, R.: \textit{G\'en\'eralit\'es sur les formes modulaires, I.} S\'eminaire Henri Cartan, tome 10, no 1, exp.\ no 7, p.\ 1--18 (1957-1958)
	
	\bibitem[God57-10]{godement57_10} Godement, R.:	\textit{S\'erie de Poincar\'e et spitzenformen}. S\'eminaire Henri Cartan, tome 10, no 1, exp. no 10, p.\ 1--38 (1957-1958)
	
	\bibitem[Hal15]{hall15} Hall, B.: \textit{Lie groups, Lie algebras, and representations. An elementary introduction.} Second edition. Grad.\ Texts in Math.\ \textbf{222}, Springer, Cham (2015)
	
	\bibitem[HarCha66]{harish66} Harish-Chandra: Discrete series for semisimple Lie groups. II. Explicit determination of the characters. \textit{Acta Math.}\ \textbf{116}, 1--111 (1966)
	
	\bibitem[HecSch76]{hecht_schmid76} Hecht, H., Schmid, W.: On integrable representations of a semisimple Lie group. \textit{Math.\ Ann.\ }\textbf{220}(2), 147--149 (1976)
	
	\bibitem[HorPitSahSch22]{horinaga_pitale_saha_schmidt22} Horinaga, S., Pitale, A., Saha, A., Schmidt, R.: The special values of the standard L-functions for $ \mathrm{GSp}_{2n}\times\mathrm{GL}_1 $. \textit{Trans.\ Amer.\ Math.\ Soc.\ }\textbf{375}(10), 6947--6982 (2022)
	
	\bibitem[Hum72]{humphreys72} Humphreys, J.~E.: \textit{Introduction to Lie algebras and representation theory.} Grad.\ Texts in Math.\ \textbf{9}, Springer-Verlag, New York--Berlin (1972)
	
	%\bibitem[Ibu12]{ibukiyama12} Ibukiyama, T.:	Vector valued Siegel modular forms of symmetric tensor weight of small degrees. \textit{Comment.\ Math. Univ.\ St.\ Pauli} \textbf{61}(1), 51--75 (2012) 
	
	\bibitem[Kli90]{klingen90} Klingen, H.:	\textit{Introductory lectures on Siegel modular forms.}	Cambridge Stud.\ Adv.\ Math.\ \textbf{20},
	Cambridge University Press, Cambridge (1990)
	
	\bibitem[Kna86]{knapp86} Knapp, A.\ W.:	\textit{Representation theory of semisimple groups. An overview based on examples}. Princeton Math.\ Ser.\ \textbf{36}, Princeton University Press, Princeton, NJ (1986)
	
	\bibitem[KnaVog95]{knapp_vogan95} Knapp, A.~W., Vogan, D.~A., Jr.: \textit{Cohomological induction and unitary representations.} Princeton Math.\ Ser.\ \textbf{45}, Princeton University Press, Princeton, NJ (1995)
	
	\bibitem[Koz00]{kozima00} Kozima, N.: On special values of standard L-functions attached to vector valued Siegel modular forms. \textit{Kodai Math.\ J.\ }\textbf{23}(2), 255--265 (2000)
	
	\bibitem[Lan85]{lang85} Lang, S.: $ \mathrm{SL}_2(\mathbb R) $. Grad.\ Texts in Math.\ \textbf{105}, Springer-Verlag, New York (1985)
	
	\bibitem[Mil77]{milicic77} Mili\v ci\'c, D.: Asymptotic behaviour of matrix coefficients of the discrete series. \textit{Duke Math.\ J.}\ \textbf{44}(1), 59--88 (1977)
	
	\bibitem[Miy06]{miyake06} Miyake, T.: \textit{Modular forms.} Translated from the 1976 Japanese original by Yoshitaka Maeda. Reprint of the first 1989 English edition. Springer Monogr.\ Math., Springer-Verlag, Berlin (2006)
	
	\bibitem[Mui09]{muic09} Mui\'c, G.:	On a construction of certain classes of cuspidal automorphic forms via Poincar\'e series. \textit{Math.\ Ann.\ }\textbf{343}(1), 207--227 (2009)
	
	\bibitem[Mui10]{muic10} Mui\'c, G.:	On the cuspidal modular forms for the Fuchsian groups of the first kind. \textit{J.\ Number Theory} \textbf{130}(7), 1488–1511 (2010) 
	
	\bibitem[Mui19]{muic19} Mui\'c, G.: Smooth cuspidal automorphic forms and integrable discrete series. \textit{Math.\ Z.\ }\textbf{292}(3--4), 895--922 (2019)
	
	\bibitem[PitSahSch21]{pitale_saha_schmidt21} Pitale, A., Saha, A., Schmidt, R.: On the standard L-function for $ \mathrm{GSp}_{2n}\times\mathrm{GL}_1 $ and algebraicity of symmetric fourth L-values for $ \mathrm{GL}_2 $. \textit{Ann.\ Math.\ Qu\'e.\ }\textbf{45}(1), 113--159 (2021)
	
	\bibitem[RamSha07]{ramakrishnan_shahidi07} Ramakrishnan, D., Shahidi, F.: Siegel modular forms of genus 2 attached to elliptic curves. \textit{Math.\ Res.\ Lett.\ }\textbf{14}(2), 315--332 (2007)
	
	\bibitem[Sie35]{siegel35} Siegel, C.\ L.: \"Uber die analytische Theorie der quadratischen Formen. \textit{Ann.\ of Math.\ (2)} \textbf{36}(3), 527--606 (1935)
	
	\bibitem[Sie39]{siegel39} Siegel, C.\ L.: Einf\"uhrung in die Theorie der Modulfunktionen n-ten Grades. \textit{Math.\ Ann.\ }\textbf{116}, 617--657 (1939)
	
	\bibitem[Tai17]{taibi17} Ta\"ibi, O.: Dimensions of spaces of level one automorphic forms for split classical groups using the trace formula. \textit{Ann.\ Sci.\ \'Ec.\ Norm.\ Sup\'er.\ (4) }\textbf{50}(2), 269--344 (2017)
	
	\bibitem[Wall88]{wallach88} Wallach, N.~R.:	\textit{Real reductive groups. I.} Pure Appl.\ Math.\ \textbf{132},	Academic Press, Inc., Boston, MA (1988)
	
	\bibitem[Wei83]{weissauer83} Weissauer, R.: Vektorwertige Siegelsche Modulformen kleinen Gewichtes. \textit{J.\ Reine Angew.\ Math.}\ \textbf{343}, 184–202 (1983)
	
	\bibitem[\v Zun19]{zunar19} \v Zunar, S.: On the non-vanishing of Poincar\'e series on the metaplectic group. \textit{Manuscripta Math.}\ \textbf{158}(1--2), 1--19 (2019)
	
	\bibitem[\v Zun23]{zunar23} \v Zunar, S.: On a family of Siegel Poincar\'e series. \textit{Int.\ J.\ Number Theory} \textbf{19}(9), 2215--2239 (2023)
	
\end{thebibliography}
\linespread{.96}

\end{document}